\documentclass[reqno, twoside,12pt,a4paper]{amsart}
\usepackage{latexsym}
\usepackage{amsmath}
\usepackage{amssymb}
\usepackage{cases}
\usepackage{mathrsfs}
\usepackage{amsthm}
\usepackage{cite}
\usepackage[usenames,dvipsnames]{color}

\newtheorem{theorem}{\textbf{Theorem}}[section]
\newtheorem{lemma}{\textbf{Lemma}}[section]
\newtheorem{proposition}{\textbf{Proposition}}[section]

\newtheorem{remark}{\textbf{Remark}}[section]
\newtheorem{definition}{\textbf{Definition}}[section]
\numberwithin{equation}{section}

\allowdisplaybreaks[4]
\title[The Equation and Boundary Condition of {C}ahn--{H}illiard Type]{Separation property and convergence to equilibrium 
\\ for the equation and dynamic boundary condition
\\ of {C}ahn--{H}illiard type with singular potential}

\author[T.\ Fukao]{Takeshi Fukao}
\address{Takeshi Fukao: Department of Mathematics, Faculty of Education,
Kyoto University of Education,
1~Fujinomori, Fukakusa, Fushimi-ku, Kyoto~612-8522 Japan.}
\email{fukao@kyokyo-u.ac.jp}

\author[H.\ Wu]{Hao Wu}
\address{Hao Wu: School of Mathematical Sciences and Shanghai Key Laboratory for Contemporary Applied Mathematics, Fudan University, Han Dan Road 220, Shanghai 200433, China;
Key Laboratory of Mathematics for Nonlinear Science (Fudan University),
Ministry of Education, Han Dan Road 220, Shanghai 200433, China.}
\email{haowufd@fudan.edu.cn}

\dedicatory{Dedicated to Professor Pierluigi Colli on the occasion of his 60th birthday with best wishes.}
\begin{document}

\begin{abstract}
We consider a class of {C}ahn--{H}illiard equation that 
models phase separation process of binary mixtures involving nontrivial boundary interactions in a 
bounded domain with non-permeable wall. The system is characterized by certain dynamic type boundary conditions 
and the total mass, in the bulk and on the boundary, is conserved for all time. 
For the case with physically relevant singular (e.g., logarithmic) potential, 
global regularity of weak solutions is established. In particular, when the spatial dimension is two, 
we show the instantaneous strict separation property such that for arbitrary positive time any weak solution stays away 
from the pure phases $\pm 1$, while in the three dimensional case, an eventual separation property for large 
time is obtained. As a consequence, we prove that every global weak solution converges to a single 
equilibrium as $t\to \infty$, by the usage of an extended {\L}ojasiewicz--Simon inequality.
\end{abstract}

\maketitle

%%%%% Section 1. %%%%%
\section{Introduction}
\setcounter{equation}{0}
In this paper, we consider the {C}ahn--{H}illiard equation:
\begin{equation}
	\begin{cases}
	\partial _t u-\Delta \mu =0, \\
	\mu = -\Delta u + F'(u),
	\end{cases}\quad {\rm in~} (0,\infty ) \times \Omega,
	\label{ogms1}
\end{equation}
subject to the following dynamic boundary conditions:
\begin{equation}
	\begin{cases}
	\partial _t u_{|_\Gamma } +\partial _{\boldsymbol{\nu }} \mu
	-\sigma \Delta_\Gamma \mu_{|_\Gamma } + \kappa \mu_{|_\Gamma } =0, \\
	\mu_{|_\Gamma }= \partial _{\boldsymbol{\nu }} u
	- \chi \Delta_\Gamma  u_{|_\Gamma } + F_\Gamma' (u_{|_\Gamma }),
	\end{cases} \quad {\rm on~} (0,\infty) \times \Gamma,
	\label{ogms2}
\end{equation}
and the initial condition
\begin{equation}
	u(0)=u_{0} \quad {\rm in~}\Omega.
	\label{ogms3}
\end{equation}
Here, $\Omega \subset \mathbb{R}^n$ ($n=2,3$) is a bounded domain with smooth boundary $\Gamma :=\partial \Omega $, $\boldsymbol{\nu }=\boldsymbol{\nu }(x)$ is the unit outer normal vector on $\Gamma$ and $\partial _{\boldsymbol{\nu }}$ denotes
the outward normal derivative on the boundary. The symbol $\Delta$ denotes the usual Laplace operator in $\Omega$ and $\Delta _{\Gamma }$ stands for the {L}aplace--{B}eltrami operator on $\Gamma$. $F$ and $F_\Gamma$ denote the bulk and boundary potentials, respectively. The constant $\kappa\geq 0$ is related to the mass exchange to the environment and $\sigma$, $\chi$ are some given nonnegative constants that account for possible boundary diffusion. When $\sigma,\,\chi>0$, system \eqref{ogms1}--\eqref{ogms3} can be regarded as equation and dynamic boundary condition of {C}ahn--{H}illiard type.

The {C}ahn--{H}illiard equation is a fundamental
diffuse interface model for multi-phase systems. It was first proposed in materials science to describe the pattern formation evolution of micro-structures during the phase separation process in binary alloys \cite{CH58,Nov} and has been extended to many areas of
scientific research, for instance, diblock
copolymer, image inpainting, and multiphase fluid flows. When the evolution is confined in a bounded domain $\Omega$, suitable boundary conditions should be taken into account for equation \eqref{ogms1}. Classical choices are the
homogeneous Neumann boundary conditions: 
\begin{equation*}
	\partial _{\boldsymbol{\nu }} \mu= \partial _{\boldsymbol{\nu }} u=0,
	 \quad {\rm on~} (0,\infty) \times \Gamma.
\end{equation*}
The corresponding initial boundary value problem of Cahn--Hilliard equation has been well-understood and rather complete results on its mathematical analysis (well-posedness, regularity of solutions and long-time behavior) have been obtained in the literature. We refer to, for instance,  \cite{AW07,EZ86,GGM17,KNP95,Kub12,RH99} and the references therein, for further details, see the recent review paper \cite{Mil17}. 

In recent studies, the so-called dynamic boundary conditions have been proposed in order to describe certain effective short-range interactions between the mixture and the solid wall (i.e., the boundary) \cite{Fis97,Kenz01}. In this case, the evolution of binary mixtures is characterized by the total free energy of the following typical form:
\begin{equation}
	E(u)
	:=
	 \int_{\Omega }^{} \left(\frac{1}{2}|\nabla u|^2+F(u) \right) dx
	+
	 \int_{\Gamma}^{} \left(\frac{\chi}{2}|\nabla_\Gamma u |^2 + F_\Gamma(u)\right) d\Gamma,
	\label{eneF}
\end{equation}
that is, the sum of a Ginzburg--Landau (bulk) free energy and of a surface free energy. The potential $F$ usually has a double-well structure and a thermodynamically relevant case is given by the so-called logarithmic potential:
\begin{equation}
F(r) = (1+r)\ln(1+r)+(1-r)\ln (1-r)-cr^2, \quad  {\rm for}~ r \in (-1,1),\label{log}
\end{equation}
where the constant $c > 0$ is large enough such that $F$ is nonconvex and has local minima  at $r=\pm r_*$, where $-1<-r_*<0<r_*<1$. This potential function is viewed as a singular one since its derivative $F':=\beta+\pi$ with 
\begin{equation}
\beta (r):=\ln \left( \frac{1+r}{1-r} \right),
\quad {\rm for~all~} r\in (-1,1),
\quad \pi (r)=-2cr,
\quad {\rm for~all~} r\in [-1,1],
\label{log1}
\end{equation}
satisfies $\lim_{r\to\pm 1} \beta(r){\rm sign} r= \infty$. 
In applications, it is often approximated by regular potentials with the prototype given by 
$F(r)=(1/4)(r^2-1)^2$ on the extended domain $\mathbb{R}$. Based on the energy functional \eqref{eneF}, different types of dynamic boundary conditions for the Cahn--Hilliard equation have been derived and analyzed in the literature, see for instance,  \cite{CFP06,CGS18,FYW17,GMS09,GS19,LW19,MZ05,PW06,WZ04}. In particular, concerning the dynamic boundary condition \eqref{ogms2} we are going to investigate in this paper, it was first introduced in \cite{Gal06} (with $\sigma=0$, $\chi>0$, $\kappa>0$, referred to as the Wentzell boundary condition) and then derived in a slightly different form by \cite{GMS11} (with $\sigma\geq 0$, $\chi\geq 0$, $\kappa=0$). This boundary condition describes bulk-surface phase separation process in a binary mixture confined to a bounded region with porous walls such that possible mass fluxes between the bulk and the boundary are allowed. The parameter $\kappa$ distinguishes the cases of permeable wall ($\kappa>0$) and non-permeable wall ($\kappa=0$), which is related to the property on conservation of total (i.e., bulk plus
boundary) mass such that 
\begin{equation}
	\frac{d}{dt}\left(\int_\Omega u dx+\int_\Gamma u d\Gamma\right)
	=-\kappa\int_\Gamma \mu d\Gamma,\quad \text{on } (0,\infty).
\label{ffmass}
\end{equation}
On the other hand, under condition \eqref{ogms2}, the system preserves the dissipation of total free energy $E(u)$ provided that $\sigma, \kappa\geq 0$:
\begin{equation}
	\frac{d}{dt} E(u) + \int_\Omega |\nabla \mu|^2 dx 
	+ \int_\Gamma \big(\sigma|\nabla_\Gamma \mu|^2 +\kappa |\mu|^2\big) d\Gamma=0, 
	\quad \text{on } (0,\infty).
	\label{ffdiss}
\end{equation}
The initial boundary value problem with regular potentials $F$ and $F_\Gamma$ has been studied extensively in the literature. When $\sigma=0$, $\chi>0$, existence, uniqueness and regularity of solutions were proved in \cite{Gal06,Gal06b,Kaj18} ($\kappa>0$) and \cite{Gal08,Kaj18} ($\kappa=0$) by different approaches; 
long-time behavior of global solutions were investigated in \cite{Gal06b, Wu07} ($\kappa>0$) 
and \cite{Gal08,GW08} ($\kappa=0$), 
proving the existence of global and exponential attractors as well as convergence of global solutions to single steady states as 
$t\to \infty$. Concerning the problem with general (singular) potentials and non-permeable wall ($\kappa=0$), existence and uniqueness of
global weak solutions and their long-time behavior were studied in \cite{CGM13} ($\sigma\geq 0$, $\chi\geq 0$) and \cite{GMS11} ($\sigma=\chi=1 $), see also \cite{CF15} in which the double
obstacle potential was handled and recent works \cite{CGSa18,CGS18,GS19} for the system with additional convection and viscous terms. Last but not least, we refer to \cite{CP14,FYW17} for numerical studies, to \cite{FY17} for the associated optimal boundary control problem, and to \cite{Mot18} for the existence of time periodic solutions.

In this paper, we consider the problem \eqref{ogms1}--\eqref{ogms3} with $\kappa=0$ and $\chi>0$ 
(taking $\chi=1$ without loss of generality), 
namely, imposing the evolution problem in a bounded domain with 
non-permeable wall and keeping the contributation of boundary diffusion in the free energy. 
In particular, we are interested in the regularity of global weak solutions and their long time behavior when the potential 
$F$ is allowed to be singular (e.g., \eqref{log}). 
As it has been pointed out in \cite{MZ10,CGM13}, 
the Cahn--Hilliard equation with dynamic boundary condition and singular potential is mathematically difficult, 
since the interplay between them may allow the solution to reach the pure states 
$\pm 1$ in regions with nonzero measure. To handle this, several attempts have been  made in the literature. 
In \cite{GMS11}, the authors obtained the 
regularity and long time behavior of solutions under certain growth restrictions on $F$, which unfortunately exclude the thermodynamically relevant logarithmic function. Later in \cite{CGM13}, the authors introduced a variational inequality (cf.\ also \cite{MZ10}) which enables them to prove the 
existence of finite-dimensional attractors for variational solutions, also in the case of logarithmic nonlinearities. 
We note that in those works, the boundary potential $F_\Gamma$ was assume to be a $C^2$ function with at most quadratic growth. On the other hand, under a different assumption that the boundary potential $F_\Gamma$ somehow dominates the bulk potential $F$  (cf.\ also  \cite{CGS14}), the authors of \cite{CF15} could prove the existence of global weak as well as strong solutions for a general class of nonlinearities. 

Below we choose to work with singular potentials in a setting similar to  \cite{CF15}. In this case, the bulk and boundary potentials in \eqref{eneF} are decomposed as 
\begin{equation*}
	F(r)=\widehat{\beta}(r)+\widehat{\pi}(r),\quad F_\Gamma(r)=\widehat{\beta}_\Gamma(r)+\widehat{\pi}_\Gamma(r),
\end{equation*}
where $\widehat{\beta},\ \widehat{\beta}_\Gamma: \mathbb{R}\to [0, \infty]$ are some convex, proper and l.s.c.\ functions and $\widehat{\pi},\ \widehat{\pi}_\Gamma: \mathbb{R}\to\mathbb{R}$ are of class $C^2$ with Lipschitz continuous first derivatives. The associated subdifferentials are denoted by $\beta=\partial \widehat{\beta}$, $\beta_\Gamma=\partial \widehat{\beta}_\Gamma$, respectively, which are maximal monotone operators with domains $D(\beta)$, $D(\beta_\Gamma)$. Under suitable assumptions on these nonlinearities (see (A1)--(A3) in Section 2 for details) that in particular are fulfilled by the physically relevant logarithmic potential \eqref{log}, the following results can be established for problem \eqref{ogms1}--\eqref{ogms3}.

\begin{itemize}
\item[(I)] \textit{Regularity of global weak solutions}. More precisely, we show the so-called strict
separation property provided that the initial datum is not a pure state $\pm 1$ 
(see Theorem \ref{septhm}): in both two and three dimensions, 
the global weak solution will be regular and stay uniformly away from $\pm 1$ after 
a sufficient large time; while in dimension two, 
the strict separation indeed happens instantaneously, 
with a uniform distance (with respect to the initial energy and total mass) 
for all $t \ge \eta$ ($\eta>0$ is an arbitrary but fixed constant). 
Our result gives a first example on the instantaneous separation property of 
weak solutions to the Cahn--Hilliard equation subject to dynamic boundary conditions in two dimension. 
It also extends the existing literature, 
for instance, \cite{MZ04, GGM17} for Cahn--Hilliard type equations with logarithmic potential as well as classical 
Neumann boundary conditions, and \cite{GMS11} for the case with 
dynamic boundary condition in which the eventually separation property 
was obtained under certain stronger assumptions on the bulk potential that excludes \eqref{log}. 
\item[(II)] \textit{Long time behavior}. Once the separation property is proven, 
the Cahn--Hilliard equation with singular potentials can be regarded as an equation with 
globally Lipschitz nonlinearities from a certain time on. 
Thus, we are able to study the long-time behavior of solutions just like the case 
with regular potentials \cite{GW08,Wu07}. 
More precisely, assuming in addition that the potentials $F$, $F_\Gamma$ are real analytic, 
we prove the convergence of any global weak solution to a single equilibrium as $t\to \infty$ (see Theorem \ref{convthm}). The same subject was treated in \cite[Theorem~3.22]{GMS11} by applying an extended {\L}ojasiewicz--{S}imon inequality.
However, the result therein was obtained only on a restricted situation for
$F$, excluding the logarithmic potential \eqref{log} (see \cite[Remark~3.8]{GMS11}). The proof of convergence to equilibrium relies on the celebrated {\L}ojasiewicz--Simon approach, see e.g., \cite{HJ99, Jen98} for a simplified illustration. It has been successfully applied to the study of Cahn--Hilliard type equations, for instance, we can refer to \cite{ASS18, CFP06, GW08, LW19, PW06, RH99, Wu07, WZ04} for the case of regular potentials and to   \cite{AW07, GGM17} for the case of the logarithmic potential \eqref{log}. See also \cite{LamWu, SW10, Wu07b}  for related results on the second order Allen--Cahn type equations under dynamic boundary conditions. 
\end{itemize}

The remaining part of this paper is organized as follows. In Section 2, we introduce the function spaces and necessary assumptions, state the main results of this paper. In Section 3, we derive some uniform estimates and a preliminary result on the regularity of global weak solutions. In Section 4, we prove our main result Theorem \ref{septhm} on the separation property. Section 5 is devoted to the proof of Theorem \ref{convthm} on the convergence to equilibrium. In the Appendix, we report some technical lemmas that have been used in this paper.

%%%%% Section 2. %%%%%
\section{Preliminaries and Main Results}
\setcounter{equation}{0}

In this section, we set up our target problem and state the main results. 

%%%%% 2.1 %%%%%
\subsection{Notation}
If $X$ is a (real) Banach space and $X^*$ is its topological dual,
then $\|\cdot\|_X$ indicates the norm of $X$ and $\langle \cdot ,\cdot \rangle_{X^*,X}$ denotes the corresponding duality product. We assume that $\Omega\subset \mathbb{R}^n$ ($n=2,3$) is a bounded domain with smooth boundary $\Gamma:=\partial \Omega$. 
Then we denote by $L^p(\Omega)$ and $L^p(\Gamma)$ $(p\geq 1)$ the standard Lebesgue spaces. When $p=2$, the inner products in the Hilbert spaces $L^2(\Omega)$ and $L^2(\Gamma)$ will be denoted by
$(\cdot, \cdot)_{L^2(\Omega)}$ and $(\cdot, \cdot)_{L^2(\Gamma)}$, respectively.
For $s\in \mathbb{R}$, $p\geq 1$, $W^{s,p}(\Omega)$ and $W^{s,p}(\Gamma)$
stand for the Sobolev spaces. If $p=2$, we denote $W^{s,p}(\Omega)=H^s(\Omega)$ and $W^{s,p}(\Gamma)=H^s(\Gamma)$.
For simplicity, we denote 
\begin{alignat*}{3}
H & :=L^{2}(\Omega ), \quad & V & :=H^1(\Omega ), \quad & W & :=H^2(\Omega ),\\
H_{\Gamma } &:=L^{2}(\Gamma ), \quad & V_\Gamma & :=H^1(\Gamma ), \quad &  W_\Gamma & :=H^2(\Gamma),
\end{alignat*}
with standard norms and inner products indicated above.
Next, we define the Hilbert spaces
\begin{align*}
\boldsymbol{H}  &:=H \times H_\Gamma,\\
\boldsymbol{V}  &:=\left\{ \boldsymbol{z} \in V \times V_\Gamma \ : \boldsymbol{z}=(z,z_\Gamma )\ \text{and}\ z_\Gamma =z_{|_\Gamma } \ {\rm a.e.\ on~} \Gamma \right\},\\
\boldsymbol{W} &:=(W \times W_\Gamma ) \cap \boldsymbol{V},
\end{align*}
endowed with natural inner products and related norms. 
Here, $z_{|_\Gamma }$ stands for the trace of the function $z$. 
Hereafter, we use a bold letter like $\boldsymbol{z}$ to denote the corresponding pair $(z, z_\Gamma)$. 
Let us restate that if $\boldsymbol{z} := (z, z_\Gamma) \in\boldsymbol{V}$ then 
$z_\Gamma$ means exactly the trace of $z$ on $\Gamma$, 
while if $\boldsymbol{z} := (z, z_\Gamma) \in \boldsymbol{H}$, 
then $z \in H$ and $z_\Gamma \in H_\Gamma$ are actually independent. 
From the definition, we easy see that $\boldsymbol{V}$ is dense in $\boldsymbol{H}$ and 
the chain of continuous embeddings holds 
$\boldsymbol{V}\subset V \subset \boldsymbol{H} \subset V^* \subset \boldsymbol{V}^*$, 
indeed, for each $z \in V$ we see that $(z,z_{|_\Gamma}) \in \boldsymbol{H}$, 
and for each $z^* \in V^*$ 
we can define $\boldsymbol{z}^* \in\boldsymbol{V}^*$ by $\langle \boldsymbol{z}^*, \boldsymbol{z} 
\rangle _{\boldsymbol{V}^*,\boldsymbol{V}}:=\langle z^*, z
\rangle _{V^*,V}$ for all $\boldsymbol{z}=(z,z_\Gamma) \in \boldsymbol{V}$. 
In what follows, we set for $\sigma\geq 0$
\begin{alignat*}{3}
	\mathcal{V}_\sigma & :=\boldsymbol{V}, \quad & \mathcal{W}_\sigma
	& := \boldsymbol{W},\quad & \text{if} \quad & \sigma>0, \\
	\mathcal{V}_\sigma & := V, \quad & \mathcal{W}_\sigma & :=W,\quad & \text{if} \quad & \sigma=0.
\end{alignat*}
For any $ \boldsymbol{z}^* \in \mathcal{V}_\sigma^*$, we define the generalized mean value by setting 
\begin{equation*} 
m( \boldsymbol{z}^*):= \frac{1}{|\Omega |+|\Gamma| }\langle \boldsymbol{z}^*, 
\boldsymbol{1}\rangle _{\mathcal{V}_{\sigma}^*,\mathcal{V}_{\sigma}},
\end{equation*}
where $|\Omega |:=\int_{\Omega }^{} 1 dx$ and $|\Gamma |:=\int_{\Gamma }^{} 1 d\Gamma $.
It leads to the usual mean value function when applied to elements of $\boldsymbol{H}$, i.e., 
$m :\boldsymbol{H} \to \mathbb{R}$ such that
\begin{equation*}
	m(\boldsymbol{z}):=\frac{1}{|\Omega |+|\Gamma| }
	\left( \int_{\Omega }^{}z dx
	+ \int_{\Gamma }^{} z_{\Gamma } d\Gamma \right),
	\quad \mbox{for all }\boldsymbol{z} \in \boldsymbol{H}.
\end{equation*}
Next, we introduce the subspace $\boldsymbol{H}_0$ of $\boldsymbol{H}$ by
\begin{equation*}
	\boldsymbol{H}_0:= \bigr\{ \boldsymbol{z} \in
	\boldsymbol{H} \ : \ m(\boldsymbol{z})=0 \bigr \},
\end{equation*}
and define $\boldsymbol{V}_0:=\boldsymbol{V} \cap \boldsymbol{H}_0$, 
$\boldsymbol{W}_0:=\boldsymbol{W}\cap \boldsymbol{H}_0$, respectively.
Then the dense and compact embedding 
$\boldsymbol{V}_ 0 \mathop{\hookrightarrow} \mathop{\hookrightarrow} \boldsymbol{H}_0$ holds (see, 
\cite[Lemma~B]{CF15}). 
Moreover, for $\sigma \geq 0$, we also set $$\mathcal{V}_{\sigma, 0} :=\mathcal{V}_{\sigma} \cap \boldsymbol{H}_0.$$ 
 The equivalent norms in $\boldsymbol{H}_0$, $\mathcal{V}_{\sigma, 0}$ are given by 
$\| \boldsymbol{z}\|_{\boldsymbol{H}_0}:=\|\boldsymbol{z}\|_{\boldsymbol{H}}$
for all
$\boldsymbol{z} \in \boldsymbol{H}_0$ and
\begin{equation*}
	\|\boldsymbol{z}\|_{\mathcal{V}_{\sigma, 0}}:=
	\left(
	\int_{\Omega }^{} |\nabla z|^2 dx
	+ \sigma
	\int_{\Gamma }^{} |\nabla _\Gamma z_\Gamma |^2  d\Gamma
	\right)^{1/2},
	\quad \mbox{for all }\boldsymbol{z} \in \mathcal{V}_{\sigma, 0},
\end{equation*}
thanks to the generalized {P}oincar\'e inequality \eqref{poin}.
For $\sigma \geq 0$, we define the following bilinear form:
\begin{equation*}
	a_\sigma( \boldsymbol{z}, \tilde{\boldsymbol{z}}) :=
	\int_{\Omega }^{} \nabla z \cdot \nabla \tilde{z} dx
	+ \sigma
	\int_{\Gamma }^{} \nabla_\Gamma  z_\Gamma  \cdot \nabla _\Gamma \tilde{z}_\Gamma
	d\Gamma,
	\quad {\rm for~all~}
	\boldsymbol{z},\tilde{\boldsymbol{z}} \in \mathcal{V}_\sigma,
\end{equation*}
and the duality mapping
$\boldsymbol{A}_\sigma: \mathcal{V}_{\sigma,0} \to \mathcal{V}_{\sigma,0}^*$ given by
\begin{equation*}
	\langle \boldsymbol{A}_\sigma
	\boldsymbol{z}, \tilde{\boldsymbol{z}}
	\rangle _{\mathcal{V}_{\sigma,0}^*, \mathcal{V}_{\sigma,0}}
	= a_\sigma( \boldsymbol{z}, \tilde{\boldsymbol{z}}),	
	\quad {\rm for~all~}
	\boldsymbol{z},\tilde{\boldsymbol{z}} \in \mathcal{V}_{\sigma,0}.
\end{equation*}
Then from \cite{Gal08} (for $\sigma=0$) and \cite{CF15} (for $\sigma>0$), we infer that the operator $\boldsymbol{A}_\sigma$ is a linear isomorphism and its inverse operator $\boldsymbol{A}_\sigma^{-1}: \mathcal{V}_{\sigma,0}^* \to \mathcal{V}_{\sigma,0}$ is compact on $\boldsymbol{H}_0$. Besides, we can define the inner product in $\mathcal{V}_{\sigma,0}^*$ by
\begin{equation}
	(\boldsymbol{z}^*,\tilde{\boldsymbol{z}}^*)_{\mathcal{V}_{\sigma,0}^*}
	:=\langle \boldsymbol{z}^*,
	\boldsymbol{A}_\sigma^{-1} \tilde{\boldsymbol{z}}^*
	\rangle _{\mathcal{V}_{\sigma,0}^*,\mathcal{V}_{\sigma,0}},
	\quad {\rm for~all~} \boldsymbol{z}^*, \tilde{\boldsymbol{z}}^*
	\in \mathcal{V}_{\sigma,0}^*.
	\label{inner}
\end{equation}
We denote by 
$\boldsymbol{P}:\boldsymbol{H} \to \boldsymbol{H}_0$ the projection 
\begin{equation*}
	\mbox{\boldmath $ P$} \mbox{\boldmath $ z$}:=
	\mbox{\boldmath $ z$}-m(\mbox{\boldmath $ z$}) \mbox{\boldmath $ 1$}
	= \bigl (z - m(\mbox{\boldmath $ z$}), z_\Gamma  - m(\mbox{\boldmath $ z$}) \bigr ),
	\quad {\rm for~all~} \boldsymbol{z} \in \boldsymbol{H}.
\end{equation*}
Then for any $\boldsymbol{z}^*\in \mathcal{V}_{\sigma,0}^*$, one can define 
$\langle \boldsymbol{z}^*, 
\boldsymbol{z} 
\rangle _{\mathcal{V}_{\sigma}^*,\mathcal{V}_{\sigma}} :=
\langle \boldsymbol{z}^*, 
\mbox{\boldmath $ P$}
\boldsymbol{z} \rangle _{\mathcal{V}_{\sigma,0}^*,\mathcal{V}_{\sigma,0}}$
for all $\boldsymbol{z} \in \mathcal{V}_{\sigma}$. 
Furthermore, one can identify the dual space $\mathcal{V}_{\sigma,0}^*$ by 
$\{\boldsymbol{z}^*\in \mathcal{V}_{\sigma}^*:\ \langle \boldsymbol{z}^*, 
\boldsymbol{1}\rangle _{\mathcal{V}_{\sigma}^*,\mathcal{V}_{\sigma}}=0\}$.

%%%%% 2.2 %%%%%
\subsection{The initial boundary value problem} 
Hereafter, for each $T\in (0, \infty)$
we denote 
$$
Q_T:=(0,T) \times \Omega,\ \ \Sigma _T:=(0,T) \times \Gamma,\ \  
Q :=(0,\infty)\times \Omega,\ \ \Sigma :=(0,\infty)\times \Gamma.
$$
The system \eqref{ogms1}--\eqref{ogms3} can be viewed as a sort of transmission problem that consists of a {C}ahn--{H}illiard equation in the bulk and another one on the boundary as a dynamic boundary condition (cf. \cite{MZ05}). To this end, introducing two new variables on $\Gamma$:
\begin{equation*}
	u_\Gamma=u_{|_\Gamma },\quad 	\mu_\Gamma=\mu_{|_\Gamma },
\end{equation*}
we can reformulate the target problem \eqref{ogms1}--\eqref{ogms3} as follows: find
$u, \mu :Q \to \mathbb{R}$ and $u_\Gamma, \mu_\Gamma:\Sigma \to \mathbb{R}$
satisfying
\begin{equation}
	\left\{
	\begin{aligned}
	&\partial _t u-\Delta \mu =0, & {\rm for~a.a.}\  (t,x)\in Q, \\
	&\mu = -\Delta u + \beta (u) + \pi (u),  & {\rm for~a.a.}\  (t,x)\in Q, \\
	&u_{|_\Gamma }=u_\Gamma,
	\quad \mu_{|_\Gamma }=\mu_\Gamma, & {\rm for~a.a.}\  (t,x)\in\Sigma, \\
	&\partial _t u_\Gamma +\partial_{\boldsymbol{\nu }} \mu
	-\sigma\Delta_\Gamma  \mu_\Gamma  =0, & {\rm for~a.a.}\ (t,x)\in \Sigma, \\
	&\mu_\Gamma  = \partial _{\boldsymbol{\nu }} u
	- \Delta_\Gamma  u_\Gamma  + \beta_\Gamma (u_\Gamma ) 
	+ \pi_\Gamma (u_\Gamma ), &  {\rm for~a.a.}\  (t,x)\in \Sigma, \\
	& u(0)=u_{0}, &  {\rm for~a.a.}\  x\in \Omega, \\
	& u_{\Gamma} (0)=u_{0\Gamma }, &  {\rm for~a.a.}\  x\in\Gamma.
	\end{aligned}
	\right.
	\label{GMS}
\end{equation}
In this manner, the original Cahn--Hilliard equation \eqref{ogms1} subject to those nontrivial boundary conditions \eqref{ogms2} can be viewed as a bulk-surface coupled system such that the bulk unknown variables $(u, \mu)$ now satisfy (standard) nonhomogeneous Dirichlet boundary conditions that are determined through a surface evolution system for the boundary variables $(u_\Gamma, \mu_\Gamma)$.
\smallskip    

Next, we present our basic hypotheses on the nonlinear terms and initial data.
\begin{itemize}
 \item[(A1)] $\beta, \beta_\Gamma \in C^1(-1,1)$ are monotone increasing functions with
\begin{alignat*}{2}
	\lim _{r \searrow -1} \beta (r) & =-\infty,\quad & \lim _{r \searrow -1} \beta_\Gamma (r) & =-\infty,\\
	\lim _{r \nearrow 1} \beta (r) & =\infty, \quad & \lim _{r \nearrow 1} \beta_\Gamma (r) & =\infty.
\end{alignat*}
Their primitive denoted by 
$\widehat{\beta }$, $\widehat{\beta }_\Gamma$, respectively, satisfy 
$\widehat{\beta }$, $\widehat{\beta }_\Gamma\in C^0([-1,1])\cap C^2(-1,1)$. 
The derivatives $\beta'$, $\beta_\Gamma'$ are convex and 
\begin{equation*}
	\beta'(r)\geq \gamma,\quad \beta_\Gamma'(r)\geq \gamma,\quad \text{for all}\ r\in (-1,1).
\end{equation*}
for some positive constant $\gamma>0$. 
Without loss of generality, we set 
$\widehat {\beta} (0)=\widehat {\beta}_\Gamma (0)=\beta(0)=\beta_\Gamma(0)=0$ 
and  make the extension $\widehat{\beta}(r)=\infty$, $\widehat{\beta}_\Gamma(r)=\infty$ for $|r|>1$. 
 \item[(A2)] There exist positive constants $c_1$, $c_2$ such that
\begin{equation*}
	\bigl| \beta(r) \bigr| 
	\leq c_1
	\bigl| 
	\beta_\Gamma(r) 
	\bigr|+c_2, \quad \text{for all}\ r\in (-1,1).
\end{equation*}
 \item[(A3)] $\pi, \pi_\Gamma \in W^{1,\infty }(\mathbb{R})$ such that 
 \begin{equation*}
	 \bigl| 
	 \pi'(r) 
	 \bigr|
	 \leq L,\quad 
	 \bigl| 
	 \pi'_\Gamma(r) 
	 \bigr|
	 \leq L,\quad \text{for all }r\in \mathbb{R},
\end{equation*}
 with $L>0$ being a certain given constant.
 \item[(A4)] $\boldsymbol{u}_0:=(u_0,u_{0\Gamma }) \in \boldsymbol{V}$, $m(\boldsymbol{u}_0) =m_0$ for some constant $m_0\in (-1,1)$ and the compatibility conditions
$\widehat{\beta }(u_0) \in L^1(\Omega )$, $\widehat{\beta }_\Gamma(u_{0\Gamma }) \in L^1(\Gamma )$ hold.
\end{itemize}
\begin{remark}
The physically relevant logarithmic potential with Lipschitz perturbations \eqref{log}, serves as a typical example that satisfies assumptions $\mathrm{(A1)}$--$\mathrm{(A3)}$. The assumption $m_0 \in (-1,1)$ in {\rm (A4)} indicates that the initial datum is not allowed to be a pure state (i.e., $\pm 1$). On the other hand, if the initial datum is a pure state then no separation process will take place.
\end{remark}
\smallskip

As a preliminary result, we have the following conclusion on existence and uniqueness of global weak solutions to problem \eqref{GMS}.
\begin{proposition}[Global weak solutions] \label{gloweak}
Suppose that $\Omega\subset \mathbb{R}^n$ $(n=2,3)$ is a bounded domain with smooth boundary 
$\Gamma$ and $\sigma\geq 0$. For arbitrary $T \in (0, \infty)$, under the assumptions {\rm (A1)--(A4)}, problem \eqref{GMS} admits a global weak solution $(\boldsymbol{u}, \boldsymbol{\mu })=(u, u_\Gamma, \mu, \mu_\Gamma)$ in the following sense:
\begin{align}
	&\boldsymbol{v} \in H^1(0,T;\mathcal{V}_{\sigma,0}^*)\cap L^\infty (0,T;\boldsymbol{V}_0) \cap L^2(0,T;\boldsymbol{W}_0), \label{p1} \\
	&\boldsymbol{\mu } \in L^2(0,T;\mathcal{V}_\sigma),	\label{p2}\\
	& \bigl( \beta(u), \beta_\Gamma(u_\Gamma) \bigr) \in L^2(0,T; \boldsymbol{H}),\label{p2a}
\end{align}
with $\boldsymbol{v}=\boldsymbol{u}-m_0\boldsymbol{1}$ such that
\begin{equation}
	\bigl\langle \boldsymbol{u}'(t), \boldsymbol{z}
	\bigr\rangle_{\mathcal{V}_{\sigma}^*, \mathcal{V}_{\sigma}}
	+ a_\sigma \bigl( \boldsymbol{\mu }(t), \boldsymbol{z} \bigr) = 0,
	\quad {\it for~all~} \boldsymbol{z} \in \mathcal{V}_{\sigma}, 
	\quad {\it a.a.\ } t \in (0,T)
	\label{p3}
\end{equation}
and
\begin{alignat}{2}
	& \mu = -\Delta u + \beta (u) + \pi (u),  & & {\it a.e.\ in~} Q_T,
	\label{p4}\\
	&\mu_\Gamma  = \partial _{\boldsymbol{\nu }} u
	- \Delta_\Gamma  u_\Gamma  + \beta_\Gamma (u_\Gamma )
	+ \pi_\Gamma (u_\Gamma ), \quad & & {\it a.e.\ on~} \Sigma_T,
	\label{p5}\\
	& u(0)=u_{0}, \quad {\it a.e.\ in~} \Omega, \quad
	u_{\Gamma} (0)=u_{0\Gamma }, & & {\it a.e.\ on~} \Gamma. 
	\label{pi}
\end{alignat}	
Moreover, the function $\boldsymbol{u}$ is unique and we have
\begin{equation}
	\bigl\| \boldsymbol{u}^{(1)}(t)-\boldsymbol{u}^{(2)}(t) 
	\bigr\|_{\mathcal{V}_{\sigma,0}^*}\leq 
	C_T 
	\bigl \| 
	\boldsymbol{u}^{(1)}_0-\boldsymbol{u}^{(2)}_0 
	\bigr \|_{\mathcal{V}_{\sigma,0}^*},\quad \text{for all}\ t\in [0,T],
	\label{unique1}
\end{equation}
where for $i=1,2$, $\boldsymbol{u}^{(i)}$ is the weak solution corresponding to the initial datum $\boldsymbol{u}^{(i)}_0$, $C_T$ is a positive constant only depending on $L$ and $T$.
\end{proposition}

The proof of Proposition \ref{gloweak} with $\sigma>0$ follows the same arguments as
\cite[Theorems~2.1, 2.2 and Remarks~1, 2]{CF15}, while the case $\sigma=0$ can be treated in a similar way with minor modifications on function spaces and energy estimates. 
Here, we only sketch the strategy of its proof.
For each $\varepsilon >0$, we consider the following
viscous {C}ahn--{H}illiard system:
\begin{align}
	&\boldsymbol{v}_\varepsilon '(t)
	+
	\boldsymbol{A}_\sigma \boldsymbol{P} \boldsymbol{\mu }_\varepsilon  (t) =\boldsymbol{0}
	\quad {\rm in}~\mathcal{V}_{\sigma,0}^*,\quad {\rm for~a.a.\ } t \in (0,T),
	\label{ap1}\\
	&\boldsymbol{\mu }_\varepsilon (t)=\varepsilon \boldsymbol{v}_\varepsilon '(t) +
	\partial \varphi \bigl( \boldsymbol{v}_\varepsilon (t) \bigr)
	+ \boldsymbol{\beta }_\varepsilon \bigl( \boldsymbol{u}_\varepsilon (t) \bigr)
	+ \boldsymbol{\pi } \bigl( \boldsymbol{u}_\varepsilon (t) \bigr)
	\quad {\rm in~} \boldsymbol{H},
	\quad {\rm for~a.a.\ } t \in (0,T),
	\label{ap2}\\
	&\boldsymbol{v}_\varepsilon (0)=\boldsymbol{v}_0 \quad {\rm in~} \boldsymbol{H}_0,
	\label{ap3}
\end{align}
where $\boldsymbol{v}_\varepsilon=\boldsymbol{u}_\varepsilon -m_0 \boldsymbol{1}$,
$\boldsymbol{\pi }(\boldsymbol{z}):=(\pi (z),\pi_\Gamma (z_\Gamma ))$ and 
$\boldsymbol{\beta }_\varepsilon (\boldsymbol{z}):=
(\beta_\varepsilon (z), \beta_{\Gamma, c_1\varepsilon} (z_\Gamma ))$ such that $\beta_\varepsilon$, 
$\beta_{\Gamma, c_1\varepsilon}$ are the standard Yosida approximation of 
$\beta$, $\beta_\Gamma$, respectively (see e.g., \cite{CC13} and \cite[Section 4]{CF15}). 
In equation \eqref{ap2},
$\varphi :\boldsymbol{H}_0 \to [0,\infty ]$ is a proper,
lower semi-continuous and convex functional
defined by
\begin{equation*}
	\varphi (\mbox{\boldmath $ z$})
	:=
	\begin{cases}
	\displaystyle
	\frac{1}{2} \int_{\Omega }^{} | \nabla z |^2 dx
	+\frac{1}{2} \int_{\Gamma }^{}
	|\nabla _{\Gamma } z_{\Gamma }  |^2 d\Gamma,
	\quad
	\mbox{if }
	\mbox{\boldmath $ z$} \in \mbox{\boldmath $ V$}_0, \vspace{2mm}\\
	+\infty, \quad \mbox{otherwise}.
	\end{cases}
\end{equation*}
The subdifferential $\partial \varphi $ of
$\varphi $ is given by
$\partial \varphi (\boldsymbol{z})=(-\Delta z, \partial_{\boldsymbol{\nu }}z-\Delta _\Gamma z_\Gamma )$ with
$D(\partial \varphi )=\boldsymbol{V}_0 \cap \boldsymbol{W}$ (see, e.g., \cite[Lemma~C]{CF15}).
Then, by the abstract theory of doubly nonlinear evolution inclusions \cite{CV90},
we can prove the existence and uniqueness of an approximate solution
$\boldsymbol{v}_\varepsilon \in H^{1}(0,T;\boldsymbol{H}_0)
\cap C([0,T];\boldsymbol{V}_0) \cap L^2(0,T;\boldsymbol{W}_0)$
with $\boldsymbol{\mu }_\varepsilon \in L^2(0,T;\mathcal{V}_\sigma)$ to
problem \eqref{ap1}--\eqref{ap3}.
From the definition of $\boldsymbol{v}_\varepsilon$, we also know that
$\boldsymbol{u}_\varepsilon \in H^{1}(0,T;\boldsymbol{H})
\cap C ([0,T];\boldsymbol{V}) \cap L^2(0,T;\boldsymbol{W})$. Then one can show that the family of approximating solutions $(\boldsymbol{u}_\varepsilon, \boldsymbol{\mu }_\varepsilon )$ satisfy sufficient a priori estimates that are uniform with respect to the approximation parameter $\varepsilon$.
Hence, by taking the limit as $\varepsilon \to 0$ (up to a subsequence), the limit function $(\boldsymbol{u}, \boldsymbol{\mu })$ is indeed our target solution to problem \eqref{GMS} satisfying properties \eqref{p1}--\eqref{pi}. Uniqueness of soluition can be easily obtained by using the energy method. For further details, we refer to \cite[Sections 3,4]{CF15}.

\subsection{Main results}

We are now in a position to state the main results of this paper. The first theorem is related to the
separation property. 
\begin{theorem}[Separation from pure states]\label{septhm}
Suppose that the domain $\Omega \subset \mathbb{R}^n$ $(n=2,3)$ is bounded with smooth boundary $\Gamma$ and $\sigma\geq 0$, besides, assumptions {\rm (A1)--(A4)} are satisfied. Let $(\boldsymbol{u}, \boldsymbol{\mu })=(u, u_\Gamma, \mu, \mu_\Gamma)$ 
 be the global weak solution to problem \eqref{GMS} obtained in {\rm Proposition \ref{gloweak}}. 
 \begin{enumerate}
 \item[(1)] There exist a constant $\delta_1 \in (0,1)$ and a large time $T_1>0$ such that 
 \begin{equation}
	\bigl\| u(t) \bigr\|_{L^\infty (\Omega )} \le 1-\delta_1,
	\quad
	\bigl\| u_\Gamma (t) \bigr\|_{L^\infty (\Gamma)} \le 1-\delta_1,
	\quad {\it for~all~} t \ge T_1,
	\label{esepe}
\end{equation}
 where the constant $\delta_1$ may depend on $m_0$ but is independent of $\boldsymbol{u}_0$.
 \item[(2)] If $n=2$, $\sigma>0$ and in addition, there is a positive constant $c_0$ such that
\begin{equation}
	\bigl| \beta '(r) \bigr| \le e^{c_0 |\beta (r)|+c_0}, 
	\quad {\it for~all~}r \in (-1,1), 
	\label{A5}
\end{equation}
then for any given $\eta>0$, there exists $\delta_2\in (0,1)$ depending on $\eta$, $m_0$ and $E(\boldsymbol{u}_0)$ such that
\begin{equation}
	\bigl\| u(t) \bigr\|_{L^\infty (\Omega )} \le 1-\delta_2,
	\quad
	\bigl\| u_\Gamma (t) \bigr\|_{L^\infty (\Gamma)} \le 1-\delta_2,
	\quad {\it for~all~} t \ge \eta.
	\label{isepe}
\end{equation}
 \end{enumerate}
\end{theorem}
\smallskip
\begin{remark}
{\rm (1)} The estimate \eqref{esepe} implies that the value of $\boldsymbol{u}$ will be strictly separated from the 
pure states $\pm 1$ at least after a certain large positive time by a uniform distance. As a consequence, 
the singular potentials $\beta$, $\beta_\Gamma$ and their derivatives will no longer blow up along the evolution and they turn out to be 
Lipschitz continuous and bounded functions. 
This fact leads to further higher-order regularity of global weak solutions and will be helpful for the study of long-time 
behavior of problem \eqref{GMS}. \smallskip
\\
{\rm (2)} It is straightforward to verify that the additional assumption \eqref{A5} is 
satisfied in the case of logarithmic potential \eqref{log1}. \smallskip
\\
{\rm (3)} We note that when $\sigma>0$, the term $\sigma \Delta_\Gamma \mu_\Gamma$ accounting for boundary diffusion yields a regularizing effect on the boundary. It remains an open question whether the conclusion \eqref{isepe} still holds when $\sigma=0$. 
\end{remark}
\smallskip

Our next result concerns the long-time behavior of problem \eqref{GMS}, more precisely, 
we prove the uniqueness of asymptotic limit of any global weak solution as $t\to \infty$.
\begin{theorem}[Convergence to equilibrium]\label{convthm}
Suppose that the domain $\Omega \subset \mathbb{R}^n$ $(n=2,3)$ is bounded with
smooth boundary $\Gamma$ and $\sigma\geq 0$, besides, assumptions {\rm (A1)--(A4)} are satisfied. In addition, we assume that
 $\beta$, $\beta_\Gamma$ are real analytic on $(-1,1)$ and $\pi $, $\pi_\Gamma$ are real analytic on $\mathbb{R}$. Let $\boldsymbol{u}$ be the global weak solution to problem \eqref{GMS} obtained in Proposition \ref{gloweak}, we have 
\begin{equation*}
	\lim_{t \to \infty} \bigl\| \boldsymbol{u}(t)-\boldsymbol{u}_\infty \bigr\|_{\boldsymbol{W}}=0,
\end{equation*}
where $\boldsymbol{u}_\infty:=(u_\infty, u_{\Gamma, \infty })$ is a steady state to problem \eqref{GMS} that satisfies the nonlocal elliptic problem
\begin{equation*}
	\begin{cases}
	-\Delta u_\infty +\beta (u_\infty )+\pi (u_\infty )=\mu _\infty, \quad
	{\it a.e.\ in~}\Omega, \\
	\partial _{\boldsymbol{\nu }} u_\infty -
	\Delta _\Gamma u_{\Gamma,\infty} + 
	\beta_\Gamma (u_{\Gamma, \infty} )+ \pi_\Gamma (u_{\Gamma, \infty} )=\mu _\infty,
	\quad {\it a.e.\ on~}\Gamma, \\
	m(\boldsymbol{u}_\infty )=m_0,
	\end{cases}
\end{equation*}
with a constant $\mu_\infty$ given by 
\begin{equation*}
	\mu_\infty=
	\frac{1}{|\Omega|+|\Gamma|} 
	\left [ 
	\int_\Omega \bigl( \beta (u_\infty )+\pi (u_\infty )  \bigr) dx
	+ \int_\Gamma \bigl( \beta_\Gamma (u_{\Gamma, \infty} )
	+ \pi_\Gamma (u_{\Gamma, \infty} ) \bigr) d\Gamma 
	\right ].
\end{equation*}
Moreover, 
\begin{equation*}
	\bigl\|
	\boldsymbol{u}(t)-\boldsymbol{u}_\infty 
	\bigr\|_{\boldsymbol{W}}
	\leq C_\eta (1+t)^{-\frac{\theta^*}{1-2\theta^*}}, \quad \mbox{\it for all } t\geq \eta>0,
\end{equation*}
where $\theta^*\in (0,1/2)$ is a constant depending on $\boldsymbol{u}_\infty$, 
the positive constant $C_\eta$ may depend on 
$E(\boldsymbol{u}_0)$, $\boldsymbol{u}_\infty$, $m_0$, $\Omega$, $\Gamma$,  and $\eta$.
\end{theorem}
\begin{remark}
The results of Theorem \ref{septhm} and Theorem \ref{convthm} 
can be extended to the case  
with permeable walls (i.e., $\kappa>0$) with minor changes in function 
spaces and estimates, keeping in mind that the mass conservation property no longer 
holds (see \eqref{ffmass}) and on the other hand, there exists an extra boundary dissipation term in the energy equality (see \eqref{ffdiss}). This will compensate the generalized Poincar\'{e} inequality to recover the $H^1$-norm of $\mu$.
\end{remark}

%%%%%%%%%%%%%%%%%%%%%%%%%%%%%%%%%%%%%%%%%%%%%%%%%
\section{Regularity of Global Weak Solutions}
\setcounter{equation}{0}

In this section, we prove some basic properties and preliminary regularity results  for the global weak solution $(\boldsymbol{u},\boldsymbol{\mu })$ to problem \eqref{GMS}.

\subsection{Mass conservation and energy equality}
Hereafter, let $(\boldsymbol{u},\boldsymbol{\mu })$ be the unique global weak solution obtained in Proposition \ref{gloweak}. Taking $\boldsymbol{z}=\boldsymbol{1}$ as the test function in the weak form \eqref{p3}, we easily deduce the mass conservation property for problem \eqref{GMS} (see also \cite[Remark 2]{CF15}):
\begin{lemma} For all $t\geq 0$, it holds
\begin{equation}
	m \bigl( \boldsymbol{u}(t) \bigr)=m_0.
	\label{masscov}
\end{equation}
\end{lemma}
Next, we define the free energy of the system (recall \eqref{eneF}):
\begin{align}
	E(\boldsymbol{z})
	&:=
	\frac{1}{2} \int_{\Omega }^{} |\nabla z|^2 dx
	+
	\frac{1}{2} \int_{\Gamma}^{} |\nabla_\Gamma z_\Gamma |^2 d\Gamma
	+
	\int_{\Omega }^{} \bigl(\widehat{\beta}(z)+ \widehat{\pi}(z)\bigl) dx\nonumber\\
	&\quad\ \  	
	+
	\int_{\Gamma }^{} \bigl(\widehat{\beta}_\Gamma(z_\Gamma ) 
	+ \widehat{\pi}_\Gamma(z_\Gamma )\bigr) d\Gamma,
	\label{ene}
\end{align}
for all $\boldsymbol{z}:=(z,z_\Gamma ) \in \boldsymbol{V}$ with
$\widehat{\beta} (z) \in L^1(\Omega )$, $\widehat{\beta}(z_\Gamma) \in L^1(\Gamma )$. 
Here, the primitives $\widehat{\pi}$, $\widehat{\pi}_\Gamma$ are given by 
\begin{equation*}
	\widehat{\pi}(r)=\int_0^r \pi(s) ds,\quad \widehat{\pi}_\Gamma(r)=\int_0^r \pi_\Gamma(s) ds, \quad \text{for }r\in [-1,1].
\end{equation*}
Then we can derive a basic energy inequality for problem \eqref{GMS} that yields uniform in time estimate for global weak solutions:
\begin{lemma} \label{basicene}
For a.a.\  $t \ge 0$, it holds 
\begin{equation}
	E\bigl(\boldsymbol{u}(t) \bigr)
	+ \int_{0}^{t} \bigl \| 
	\boldsymbol{u}'(s) \bigr \|_{\mathcal{V}_{\sigma,0}^*}^2 ds
	\le E(\boldsymbol{u}_0).
	\label{ein}
\end{equation}
Then there exists a positive constant $M_1$ such that
\begin{gather}
	\| \boldsymbol{u} \|_{L^\infty (0, \infty;\boldsymbol{V})}
	+
	\int_{0}^{\infty} \bigl \| 
	\boldsymbol{u}'(s) \bigr \|_{\mathcal{V}_{\sigma,0}^*}^2 ds 
	\le M_1,
	\label{est0}
	\\
	\int_{0}^{\infty }
	\bigl \| \boldsymbol{P} \boldsymbol{\mu }(s) 
	\bigr \|_{\mathcal{V}_{\sigma,0}}^2 ds
	\le M_1.
	\label{estmu}
\end{gather}
\end{lemma}
\begin{proof}
The conclusion can be draw by working with the approximate solutions of \eqref{ap1}--\eqref{ap3} and then passing to the limit. Using the fact
$\boldsymbol{v}_\varepsilon' \in L^2(0,T;\boldsymbol{H}_0)$,
the chain rule of the subdifferential (see, e.g., \cite[Lemma 4.3, Section IV]{Sho97}) and \eqref{inner}, we see that
$\varphi (\boldsymbol{v}_\varepsilon (\cdot ))$ is absolutely continuous on $[0,T]$ and
\begin{align}
	& \int_{0}^{t}
	\bigl \| \boldsymbol{u}_\varepsilon '(s) \bigr \|_{\mathcal{V}_{\sigma,0}^*}^2
	ds \nonumber \\
	& \quad = \int_{0}^{t}
	\bigl \langle \boldsymbol{v}_\varepsilon '(s),
	\boldsymbol{A}_\sigma^{-1}\bigl(
	-\boldsymbol{A}_\sigma \boldsymbol{P} \boldsymbol{\mu }_\varepsilon (s)
	\bigr)
	\bigr \rangle_{\mathcal{V}_{\sigma,0}^*,\mathcal{V}_{\sigma,0}} ds \nonumber \\
	& \quad = -\int_{0}^{t}
	\bigl( \boldsymbol{v}_\varepsilon '(s), \boldsymbol{\mu }_\varepsilon (s)
	\bigr)_{\boldsymbol{H}} ds \nonumber \\
	& \quad = -  \int_{0}^{t}
	\bigl( \boldsymbol{v}_\varepsilon '(s), \varepsilon \boldsymbol{v}_\varepsilon '(s)
	+ \partial \varphi \bigl( \boldsymbol{v}_\varepsilon (s) \bigr)
	\bigr)_{\boldsymbol{H}_0} ds
	- \int_{0}^{t}
	\bigl( \boldsymbol{u}_\varepsilon '(s),
	\boldsymbol{\beta }_\varepsilon \bigl( \boldsymbol{u}_\varepsilon (s) \bigr)
	+ \boldsymbol{\pi}\bigl( \boldsymbol{u}_\varepsilon (s) \bigr)
	\bigr)_{\boldsymbol{H}} ds \nonumber \\
	& \quad = - \varepsilon \int_{0}^{t}
	\bigl \| \boldsymbol{v}_\varepsilon '(s)
	\bigr \|_{\boldsymbol{H}_0}^2 ds
	- \int_{0}^{t} \frac{d}{ds} E_\varepsilon\bigl( \boldsymbol{u}_\varepsilon (s) \bigr) ds	\nonumber \\
	& \quad = - \varepsilon \int_{0}^{t}
	\bigl\| \boldsymbol{v}_\varepsilon '(s)
	\bigr\|_{\boldsymbol{H}_0}^2 ds
	- E_\varepsilon\bigl( \boldsymbol{u}_\varepsilon (t) \bigr) + E_\varepsilon(\boldsymbol{u}_0),
	\label{einap}
\end{align}
for all $t \in [0,T]$, where 
\begin{align}
	E_\varepsilon(\boldsymbol{z})&:=
	\frac{1}{2} \int_{\Omega }^{} |\nabla z|^2 dx
	+
	\frac{1}{2} \int_{\Gamma}^{} |\nabla_\Gamma z_\Gamma |^2 d\Gamma
	+
	\int_{\Omega }^{} \bigl (\widehat{\beta}_\varepsilon(z)+ \widehat{\pi}(z) 
	 \bigr) dx\nonumber\\
	&\quad\ \  +
	\int_{\Gamma }^{} \bigl( 
	\widehat{\beta}_{\Gamma,\varepsilon}(z_\Gamma ) 
	+ \widehat{\pi}_\Gamma(z_\Gamma )
	\bigr) d\Gamma,
	\label{enea}
\end{align}
with $\widehat{\beta}_\varepsilon(r)=\int_0^r\beta_\varepsilon(s)ds$, $\widehat{\beta}_{\Gamma,\varepsilon}(r)=\int_0^r\beta_{\Gamma,\varepsilon}(s)ds$.

Using the fact of weak and strong convergences (see, \cite[Section~4.3]{CF15} for $\sigma>0$ and the case $\sigma=0$ can be treated similarly by changing the corresponding function spaces)
\begin{align*}
	\boldsymbol{v}_\varepsilon \to \boldsymbol{v} \quad 
	& {\rm weakly~in~}H^1(0,T;\mathcal{V}_{\sigma,0}^*),
	\quad \text{with } \boldsymbol{v}=\boldsymbol{u}-m_0 \boldsymbol{1},
	\\
	\varepsilon \boldsymbol{v}_\varepsilon \to \boldsymbol{0} \quad 
	& {\rm strongly~in~}H^1(0,T;\boldsymbol{H}_0),
	\\
	\boldsymbol{u}_\varepsilon \to \boldsymbol{u} \quad 
	& {\rm strongly~in}~C\bigl( [0,T];\boldsymbol{H} \bigr) \cap L^\infty (0,T;\boldsymbol{V}),
\end{align*}
 the lower semicontinuity of norms and the maximal monotonicity of $\beta$, $\beta_\Gamma$, we obtain
\begin{align*}
	&\liminf_{\varepsilon \to 0}
	\int_{0}^{t} 
	\bigl \| \boldsymbol{u}_\varepsilon' (s)
	\bigr \|_{\mathcal{V}_{\sigma,0}^*}^2ds
	\ge
	\int_{0}^{t} 
	\bigl \| \boldsymbol{u}'(s) 
	\bigr \|_{\mathcal{V}_{\sigma,0}^*}^2 ds, \\
	&\liminf_{\varepsilon \to 0} \varepsilon
	\int_{0}^{t} 
	\bigl \| \boldsymbol{v}_\varepsilon' (s) 
	\bigr \|_{\boldsymbol{H}_0}^2ds
	=0, \quad {\rm for~all~} t \in [0,T],\\
	&\liminf_{\varepsilon \to 0}
	E_\varepsilon\bigl( \boldsymbol{u}_\varepsilon (t) \bigr)
	=\lim_{\varepsilon \to 0}
	E_\varepsilon\bigl( \boldsymbol{u}_\varepsilon (t) \bigr)
	= E\bigl( \boldsymbol{u}(t) \bigr), 
	\quad {\rm for~a.a.\ } t \in (0,T).
\end{align*}
Therefore, taking $\liminf$ as $\varepsilon \to 0$ in \eqref{einap} (noting that $T>0$ is arbitrary), we conclude the energy inequality \eqref{ein}.
We recall that from assumptions (A1)--(A2), there exist a positive constants $c_3$ such that the primitives $\widehat{\beta }$ and $\widehat{\beta }_\Gamma$ satisfy
\begin{equation}
	\widehat{\beta }(r)+ \widehat{\pi}(r)\ge -c_3, \quad \widehat{\beta }_\Gamma(r)+\widehat{\pi}_\Gamma(r)\ge -c_3, \quad {\rm for~all~} r \in [-1,1]. 
	\label{c1c2}
\end{equation}
Hence, we have
\begin{equation}
	E\bigl( \boldsymbol{u}(t) \bigr) 
	\ge \frac{1}{2} \int_{\Omega }^{} \bigl| \nabla u(t) \bigr|^2 dx
	+
	\frac{1}{2} \int_{\Gamma}^{} \bigl| \nabla_\Gamma  u_\Gamma (t) \bigr|^2 d\Gamma
	-c_3\bigl( |\Omega | + |\Gamma | \bigr),
	\nonumber
\end{equation}
for a.a.\ $t \ge 0$. Therefore, recalling \eqref{masscov} and the generalized Poincar\'{e} inequality \eqref{poin}, we obtain
\begin{equation*}
	\int_{0}^{t} \bigl\|\boldsymbol{u}'(s) \bigr\|_{\mathcal{V}_{\sigma,0}^*}^2 ds
	\le E(\boldsymbol{u}_0)+ c_3\bigl( |\Omega | + |\Gamma | \bigr)
\end{equation*}
and 
\begin{align*}
	\bigl \| \boldsymbol{u}(t) \bigr \|_{\boldsymbol{V}}^2
	&= \bigl \| \boldsymbol{u}(t) \bigr \|_{\boldsymbol{H}}^2
	+ \int_{\Omega }^{} \bigl| \nabla u(t) \bigr|^2 dx  +
	\int_{\Gamma}^{} \bigl| \nabla_\Gamma  u_\Gamma (t) \bigr|^2 d\Gamma\nonumber\\
	&= \bigl \| \boldsymbol{u}(t) - 
	m \bigl ( \boldsymbol{u}(t) \bigr ) \bigr \|_{\boldsymbol{H}}^2
	+ \bigl( |\Omega|+|\Gamma| \bigr) \bigl | m \bigl ( \boldsymbol{u}(t) \bigr ) \bigr |^2 
	+ \bigl \| \boldsymbol{u}(t)-m \bigl( \boldsymbol{u}(t) \bigr) 
	\bigr \|_{\boldsymbol{V}_0}^2\nonumber\\
	&\leq (C_{\rm P}+1) \bigl \|\boldsymbol{u}(t)-m_0\boldsymbol{1} \bigr \|_{\boldsymbol{V}_0}^2 
	+ \bigl( |\Omega|+|\Gamma| \bigr) |m_0|^2\nonumber\\
	&\leq 2(C_{\rm P}+1)E(\boldsymbol{u}_0)+ 2(C_{\rm P}+1)c_3\bigl( |\Omega | + |\Gamma | \bigr)+ 
	\bigl( |\Omega|+|\Gamma| \bigr) |m_0|^2,
\end{align*}
for a.a.\ $t \ge 0$.
Since the right hand side is independent of $t$, then using
the {L}ebesugue monotone convergence theory, we obtain the estimate \eqref{est0}.
On the other hand, by the comparison of the equation \eqref{ap1}, we also
get \eqref{estmu}. 
\end{proof}

Next, we show that the weak solutions to problem \eqref{GMS} satisfy an energy equality, which is a standard structure of the {C}ahn--{H}illiard system.
\begin{lemma}\label{energyeq}
For any $\eta>0$, the mapping $t \mapsto E(\boldsymbol{u}(t))$ is absolutely continuous for all $t \ge \eta$ and
\begin{equation}
	\frac{d}{dt} E\bigl( \boldsymbol{u}(t) \bigr)
	+  \bigl\| \boldsymbol{u}'(t) \bigr \| _{\mathcal{V}_{\sigma,0}^*}^2
	= 0,\quad \text{for a.a.\ }t \ge \eta.
	\label{equality}
\end{equation}
\end{lemma}
\begin{proof} 
Firstly, for any given $\eta>0$, we show that there exists a positive constant $M_2$ such that
\begin{equation}
	\|\boldsymbol{u}'\|_{L^\infty (\eta,\infty ;\mathcal{V}_{\sigma,0}^*)}+
	\int_{t}^{t+1}  \bigl \| \boldsymbol{u}'(s) \bigr \|_{\boldsymbol{V}}^2 ds\le M_2,
	\quad \text{for all }t \ge \eta.
	\label{est32}
\end{equation}
Hereafter, for each $h>0$, we use the symbol of difference quotient
$\partial _t^h \boldsymbol{v}(t):=(\boldsymbol{v}(t+h)-\boldsymbol{v}(t))/h$ with respect to
the time variable $t$. Taking the difference of \eqref{ap1} at $t=s+h$ and $t=s$, we have
\begin{equation*}
	\partial _t^h \boldsymbol{v}_\varepsilon'(s)+\boldsymbol{A}_\sigma \boldsymbol{P}
	\bigl(\partial _t^h \boldsymbol{\mu }_\varepsilon (s) \bigr)=\boldsymbol{0}
	\quad {\rm in~} \mathcal{V}_{\sigma,0}^*.
\end{equation*}
This is equivalent to
\begin{equation}
	\varepsilon \partial _t^h \boldsymbol{v}_\varepsilon '(s)
	+ \boldsymbol{A}^{-1}_\sigma\partial _t^h \boldsymbol{v}_\varepsilon '(s)
	+\partial \varphi \bigl( \partial _t^h \boldsymbol{v}_\varepsilon (s)  \bigr)
	= \boldsymbol{P} \bigl( - \partial _t^h \boldsymbol{\beta }_\varepsilon\bigl( \boldsymbol{u}_\varepsilon (s) \bigr) -  \partial _t^h \boldsymbol{\pi }\bigl( \boldsymbol{u}_\varepsilon (s) \bigr) \bigr)
	\quad {\rm in~} \boldsymbol{H}_0,
	\label{232}
\end{equation}
for a.a.\ $s \ge 0$. Multiplying $\partial _t^h \boldsymbol{v}_\varepsilon (s) \in \boldsymbol{H}_0$ by \eqref{232},
using \eqref{inner}, and applying the Ehrling  lemma given by Lemma~A.1, we obtain that
\begin{align}
	& \frac{\varepsilon }{2} \frac{d}{ds}
	\bigl\| \partial _t^h \boldsymbol{v}_\varepsilon (s) \bigr\|_{\boldsymbol{H}_0}^2
	+ \frac{1}{2} \frac{d}{ds}
	 \bigl\| \partial _t^h \boldsymbol{v}_\varepsilon (s)  \bigr\|_{\mathcal{V}_{\sigma,0}^*}^2
	+
	 \bigl\| \partial _t^h \boldsymbol{v}_\varepsilon (s) \bigr\|_{\boldsymbol{V}_0}^2
	\nonumber \\
	& \quad \le \frac{1}{h^2}
	\bigl( \pi \bigl( u_\varepsilon (s+h) \bigr)
	-\pi \bigl( u_\varepsilon(s) \bigr),
	u_\varepsilon (s+h) - u_\varepsilon(s) \bigr)_H
	\nonumber \\
	& \quad \quad { }+
	 \frac{1}{h^2}
	\bigl( \pi_\Gamma \bigl( u_{\Gamma, \varepsilon} (s+h) \bigr)
	-\pi_\Gamma \bigl( u_{\Gamma, \varepsilon}(s) \bigr),
	u_{\Gamma, \varepsilon}(s+h) - u_{\Gamma, \varepsilon}(s) \bigr)_{H_\Gamma }
	\nonumber \\
	& \quad \le L  \bigl\| \partial _t^h \boldsymbol{v}_\varepsilon (s)  \bigr\|_{\boldsymbol{H}_0}^2
	\nonumber \\
	& \quad \le
	\frac12 \bigl\| \partial _t^h \boldsymbol{v}_\varepsilon (s) \bigr\|_{\boldsymbol{V}_0}^2
	+
	 M_2' \bigl\| \partial _t^h \boldsymbol{v}_\varepsilon (s) \bigr\|_{\boldsymbol{V}_0^*}^2\nonumber\\
	&\quad  \le
	\frac12 \bigl\| \partial _t^h \boldsymbol{v}_\varepsilon (s) \bigr\|_{\boldsymbol{V}_0}^2
	+
	M_2'' \bigl\| \partial _t^h \boldsymbol{v}_\varepsilon (s) \bigr\|_{\mathcal{V}_{\sigma,0}^*}^2,\quad \text{for a.a.\ }s \ge 0,
	\label{gron}
\end{align}
since $\beta $, $\beta_\Gamma$ are monotone. 
The constants $M_2'$ and $M_2''$ in \eqref{gron} may depend on $L$. Then for any fixed $\eta>0$, applying the uniform {G}ronwall type inequality given by
Lemma~A.2 with $t_0:=0$ and $r:=\eta$, we deduce
\begin{align}
	& \varepsilon
	 \bigl \| \partial _t^h \boldsymbol{v}_\varepsilon (t+\eta) \bigr \| _{\boldsymbol{H}_0}^2
	+
	\bigl \|  \partial _t^h \boldsymbol{v}_\varepsilon (t+\eta) \bigr \| _{\mathcal{V}_{\sigma,0}^*}^2
	\nonumber \\
	& \quad \le\frac{1}{\eta}
	\left( \varepsilon
	\int_{t}^{t+\eta} \bigl \|  \partial _t^h \boldsymbol{v}_\varepsilon (s)
	 \bigr \| _{\boldsymbol{H}_0}^2 ds+ 
	\int_{t}^{t+\eta} \bigl \|  \partial _t^h
	\boldsymbol{v}_\varepsilon (s) \bigr \| _{\mathcal{V}_{\sigma,0}^*}^2 ds
	\right)
	e^{2  M_2'' \eta},\quad \text{for all }t \ge 0.
	\label{key}
\end{align}
On the other hand, from the regularity of
$\boldsymbol{v}'_\varepsilon $, we have already seen that
\begin{align*}
	\sqrt{\varepsilon } \partial _t^h \boldsymbol{v}_\varepsilon
	\to \sqrt{\varepsilon } \boldsymbol{v}'_\varepsilon
	\quad & {\rm in~}L^2(t,t+\eta;\boldsymbol{H}_0),
	\\
	\partial _t^h \boldsymbol{v}_\varepsilon
	\to  \boldsymbol{v}'_\varepsilon
	\quad & {\rm in~}L^2(t,t+\eta;\mathcal{V}_{\sigma,0}^*)
\end{align*}
as $h \to 0$, that is, for each $t \ge 0$ and $\alpha >0$, there
exists $h^*:=h^*(\alpha, t)>0$ such that
\begin{equation*}
	\varepsilon
	\int_{t}^{t+\eta} \bigl \|  \partial _t^h \boldsymbol{v}_\varepsilon (s)
	- \boldsymbol{v}'_\varepsilon (s)
	 \bigr \| _{\boldsymbol{H}_0}^2 ds<\frac{\eta\alpha }{4}, 
	\quad
	\int_{t}^{t+\eta} \bigl \|  \partial _t^h
	\boldsymbol{v}_\varepsilon (s)
	-
	\boldsymbol{v}'_\varepsilon (s) \bigr \| _{\mathcal{V}_{\sigma,0}^*}^2 ds
	<\frac{\eta\alpha }{4},
\end{equation*}
for all $h\in(0,h^*)$. 
Moreover, using \eqref{einap} and  \eqref{c1c2}  we have in the right hand side of \eqref{key}
\begin{align}
	&
	\frac{\varepsilon}{\eta}
	\int_{t}^{t+\eta} \bigl \| \partial _t^h \boldsymbol{v}_\varepsilon (s)
	 \bigr \| _{\boldsymbol{H}_0}^2 ds + \frac{1}{\eta}
	\int_{t}^{t+\eta} \bigl \| \partial _t^h
	\boldsymbol{v}_\varepsilon (s)  \bigr \| _{\mathcal{V}_{\sigma,0}^*}^2 ds
	\nonumber \\
	&
	\quad \le
	\frac{2\varepsilon}{\eta}
	\int_{t}^{t+\eta} \bigl \|  \boldsymbol{v}'_\varepsilon (s)
	 \bigr \| _{\boldsymbol{H}_0}^2 ds
	+ \frac{2}{\eta}
	\int_{t}^{t+\eta} \bigl \|
	\boldsymbol{v}_\varepsilon' (s)  \bigr \|_{\mathcal{V}_{\sigma,0}^*}^2 ds
	+\alpha
	\nonumber \\
	& \quad \le
	\frac{2}{\eta} \Bigl[
	E_\varepsilon(\boldsymbol{u}_0) -E_\varepsilon\bigl( \boldsymbol{u}_\varepsilon (t+\eta) \bigr)
	\Bigr] +\alpha
	\nonumber \\
	& \quad \le
	\frac{2}{\eta}
	\Bigl[ 
	E_\varepsilon(\boldsymbol{u}_0)	+ c_3\bigl( |\Omega | + |\Gamma | \bigr) 
	\Bigr] 
	+\alpha,
		\label{est1}
\end{align}
for all $h\in(0,h^*)$.
Recalling \eqref{key} and changing the variable $\tau :=t+\eta$
\begin{align*}
	&\varepsilon
	 \bigl \| \partial _t^h \boldsymbol{v}_\varepsilon (\tau )  \bigr \|_{\boldsymbol{H}_0}^2
	+
	 \bigl \| \partial _t^h \boldsymbol{v}_\varepsilon (\tau )  \bigr \|_{\mathcal{V}_{\sigma,0}^*}^2\nonumber\\
	&\quad \le \left\{ \frac{2}{\eta}
	\Bigl[E_\varepsilon(\boldsymbol{u}_0)	+ c_3\bigl( |\Omega | + |\Gamma | \bigr)\Bigr]
	+\alpha  \right\}
	e^{2  M_2''\eta},\quad \text{for all }\tau \ge \eta.
\end{align*}
for all $h\in(0,h^*)$, where the right hand side is independent of $h$.
Thus, letting $h \to 0$, we see that
\begin{align*}
	\sqrt{\varepsilon } \partial _t^h \boldsymbol{v}_\varepsilon
	\to \sqrt{\varepsilon } \boldsymbol{v}'_\varepsilon
	&\quad {\rm weakly~star~in~}L^\infty (\eta,\infty ;\boldsymbol{H}_0),
	\\
	\partial _t^h \boldsymbol{v}_\varepsilon
	\to  \boldsymbol{v}'_\varepsilon
	&\quad {\rm weakly~star~in~}L^\infty (\eta,\infty ;\mathcal{V}_{\sigma,0}^*)
\end{align*}
as $h \to 0$ with the following estimate
\begin{equation}
	\varepsilon
	 \bigl \| \boldsymbol{v}_\varepsilon' (\tau )  \bigr \|_{\boldsymbol{H}_0}^2
	+
	 \bigl \| \boldsymbol{v}_\varepsilon' (\tau )  \bigr \|_{\mathcal{V}_{\sigma,0}^*}^2
	\le \left\{\frac{2}{\eta}
	\Bigl[E_\varepsilon(\boldsymbol{u}_0)	+ c_3\bigl( |\Omega | + |\Gamma | \bigr)\Bigr]
	+\alpha  \right\}
	e^{2  M_2''\eta}
	\label{bound}
\end{equation}
for all $\tau \ge \eta$. Since $\lim_{\varepsilon\to 0}E_\varepsilon(\boldsymbol{u}_0)=E(\boldsymbol{u}_0)$, 
the right hand side of \eqref{bound} 
can be bounded by a constant $M_2'''$ independent of $\varepsilon $. This implies that as $\varepsilon \to 0$ it holds 
\begin{align*}
	\sqrt{\varepsilon } \boldsymbol{v}'_\varepsilon
	\to \boldsymbol{0} \quad 
	& {\rm strongly~in~}L^\infty (\eta,\infty ;\boldsymbol{H}_0),
	\\
	\boldsymbol{v}'_\varepsilon
	\to \boldsymbol{v}' \quad 
	&{\rm weakly~star~in~}L^\infty (\eta,\infty ;\mathcal{V}_{\sigma,0}^*)
\end{align*}
 with the same estimate
\begin{equation}
	 \bigl \| \boldsymbol{u}' (\tau )  \bigr \|_{\mathcal{V}_{\sigma,0}^*}^2
	=
	\bigl \| \boldsymbol{v}' (\tau )  \bigr \|_{\mathcal{V}_{\sigma,0}^*}^2
	\le  M_2''',\quad \text{for all }\tau \ge \eta.
	\label{fe}
\end{equation}
Finally, integrating
\eqref{gron} over $[t,t+1]$ with respect to time $s$,
we get for each $t \ge \eta$
\begin{align*}
	 \int_{t}^{t+1}
	 \bigl \| \partial _t^h \boldsymbol{v}_\varepsilon (s)  \bigr \|_{\boldsymbol{V}_0}^2
	ds
	&  \le
	\varepsilon
	 \bigl \| \partial _t^h \boldsymbol{v}_\varepsilon (t) \bigr \|_{\boldsymbol{H}_0}^2
	+
	 \bigl \| \partial _t^h \boldsymbol{v}_\varepsilon (t)  \bigr \|_{\mathcal{V}_{\sigma,0}^*}^2
	+
	2   M_2''\int_{t}^{t+1}
	 \bigl \| \partial _t^h \boldsymbol{v}_\varepsilon (s)  \bigr \|_{\mathcal{V}_{\sigma,0}^*}^2 ds
	\nonumber \\
	&  \le \left\{ \Bigl( \frac{2}{\eta} + 2   M_2'' \Bigr) 
	 \Bigl[ E_\varepsilon(\boldsymbol{u}_0)
	+c_3\bigl( |\Omega | + |\Gamma | \bigr) \Bigr]+\alpha  \right\}
	e^{2   M_2''\eta}
\end{align*}
for all $h\in(0,h^*)$.
Thus, letting $h \to 0$ and $\varepsilon \to 0$ again, 
we see that
\begin{equation}
	\int_{t}^{t+1}
	 \bigl \| \boldsymbol{u}' (s)  \bigr \|_{\boldsymbol{V}}^2
	ds 
	\le C_{\rm P}
	\int_{t}^{t+1}
	 \bigl \| \boldsymbol{v}' (s)  \bigr \|_{\boldsymbol{V}_0}^2
	ds \le C_{\rm P}\left( \frac{2}{\eta} + 2   M_2''\right)    M_2''',
	\quad \text{for all } t \ge \eta.
	\label{se}
\end{equation}
 Combining \eqref{fe} and \eqref{se} we get the estimate \eqref{est32}.

By virtue of the {S}obolev embedding theorem, we infer from \eqref{est0} and \eqref{est32} the additional continuity 
\begin{equation}
	\boldsymbol{u} \in C  \bigl( [\eta,\infty  ) ;\boldsymbol{V}\bigr).
	\label{addconti}
\end{equation}

Hence, the mapping $t \mapsto E(\boldsymbol{u}(t))$ is continuous for all $t \ge \eta$.
Moreover, going back to the proof of Lemma~\ref{basicene}, we have
\begin{align*}
	&-\int_{s}^{t}
	 \bigl \| \boldsymbol{u} '(\tau )  \bigr \|_{\mathcal{V}_{\sigma,0}^*}^2
	d\tau\\
	& \quad = \int_{s}^{t}
	\bigl( \boldsymbol{v} '(\tau ), \boldsymbol{\mu } (\tau )
	\bigr)_{\boldsymbol{H}} d\tau  \nonumber \\
	& \quad  =  \int_{s}^{t}
	\bigl( \boldsymbol{v} '(\tau ),
	\partial \varphi \bigl( \boldsymbol{v}(\tau ) \bigr)
	\bigr)_{\boldsymbol{H}_0} d\tau
	+ \int_{s}^{t}
	\bigl( \boldsymbol{u}'(\tau ),
	\boldsymbol{\beta }\bigl( \boldsymbol{u} (\tau ) \bigr)
	+ \boldsymbol{\pi}\bigl( \boldsymbol{u} (\tau ) \bigr)
	\bigr)_{\boldsymbol{H}} d\tau  \nonumber \\
	& \quad =
	\int_{s}^{t} \frac{d}{d\tau } E\bigl( \boldsymbol{u} (\tau) \bigr) d\tau
	\nonumber \\
	&\quad  = E\bigl( \boldsymbol{u} (t) \bigr)
	- E\bigl( \boldsymbol{u} (s) \bigr)
	\quad {\rm for~all~} s,t \ge \eta,
\end{align*}
that is, the mapping
$t \mapsto E(\boldsymbol{u}(t))$ is absolutely continuous for all $t \ge \eta$ and
the energy equality \eqref{equality} holds for a.a.\  $t \ge \eta$. 
\end{proof}

\subsection{Higher-order estimates} 
We proceed to derive some higher-order estimates for the weak solution.

\begin{lemma} \label{muV}
For any $\eta>0$, there exists a positive constant $M_3$ such that
\begin{align}
	&  \bigl \| \beta (u)  \bigr \|_{L^\infty (\eta,\infty;L^1(\Omega ))}
	+  \bigl \| \beta_\Gamma (u_\Gamma )  \bigr \|_{L^\infty (\eta,\infty;L^1(\Gamma))} \le M_3,
	\nonumber \\
	& \| \boldsymbol{\mu }\|_{L^\infty (\eta,\infty;\mathcal{V}_\sigma)} \le M_3.
	\label{est33}
\end{align}
\end{lemma}

\begin{proof}
 From the comparison in equation \eqref{ap1} with \eqref{est32},
we see that
\begin{equation}
	\| \boldsymbol{P} \boldsymbol{\mu }
	\|_{L^\infty (\eta,\infty; \mathcal{V}_{\sigma,0})} =
	\|  \boldsymbol{v}' \|_{L^\infty (\eta,\infty; \mathcal{V}_{\sigma,0}^*)} \le M_2.
	\label{estmu2}
\end{equation}
Next, we estimate the mean value  $m(\boldsymbol{\mu })$ for the chemical potential $\boldsymbol{\mu }$.
Taking $\boldsymbol{z}:=\boldsymbol{A}_\sigma^{-1}\boldsymbol{v}(s)$ in the weak form \eqref{p3} at $t=s$, using
\eqref{inner} and \eqref{p4}--\eqref{p5} we get
\begin{align}
	&\bigl( \boldsymbol{v}'(s), \boldsymbol{v}(s)
	\bigr)_{\mathcal{V}_{\sigma,0}^*}
	+ \bigl \|\boldsymbol{v}(s)  \bigr \|_{\boldsymbol{V}_0}^2
	+ \bigl( \boldsymbol{\beta }\bigl( \boldsymbol{u}(s) \bigr), \boldsymbol{u}(s)-m_0 \boldsymbol{1}
	\bigr)_{\boldsymbol{H}}+ \bigl( \boldsymbol{\pi }\bigl( \boldsymbol{u}(s) \bigr), \boldsymbol{v}(s)\bigr)_{\boldsymbol{H}}= 0,
	\label{uvstar}
\end{align}
for a.a.\ $s\ge 0$. 
From assumption (A1) and (A4), we see that there exists positive constants
$c_4 $ and $c_5$ such that (cf.\ \cite[Proposition A.1]{MZ04}, also \cite[Section~5]{GMS09})
\begin{alignat}{2}
	\beta (r)(r-m_0) & \ge c_4 \bigl| \beta (r) \bigr| -c_5 \quad & {\rm for~all~} & r \in (-1,1),
	\label{gms}\\
	\beta_\Gamma (r)(r-m_0) & \ge c_4 \bigl| \beta_\Gamma (r) \bigr| -c_5 \quad & {\rm for~all~} & r \in (-1,1).\label{gmsa}
\end{alignat}
Then from \eqref{gms}--\eqref{gmsa} we deduce from \eqref{uvstar} and (A2) that
\begin{align*}
	& c_4 \int_{\Omega }^{} \bigl| \beta \bigl( u(s) \bigr)\bigr| dx
	+ c_4 \int_{\Gamma}^{} \bigl| \beta_\Gamma \bigl( u_\Gamma (s) \bigr)\bigr| d\Gamma
	\\
	& \quad
	\le c_5\bigl( |\Omega |+|\Gamma |\bigr)
	+  \bigl\| \boldsymbol{\pi }
	\bigl( \boldsymbol{u}(s) \bigr)  \bigr\|_{\boldsymbol{H}}
	 \bigl\| \boldsymbol{v}(s)
	 \bigr\|_{\boldsymbol{H}_0}
	+ \bigl\| \boldsymbol{v}'(s)  \bigr\|_{\mathcal{V}_{\sigma,0}^*}
	 \bigl\|\boldsymbol{v}(s)
	 \bigr\|_{\mathcal{V}_{\sigma,0}^*} \\
	& \quad \le c_5\bigl( |\Omega |+|\Gamma |\bigr)
	+ \bigl( L \bigl\|
	\boldsymbol{u}(s)  \bigr\|_{\boldsymbol{H}}
	+ \bigl\|\boldsymbol{\pi }(\boldsymbol{0})\bigr\|_{\boldsymbol{H}}\bigr)
	 \Bigl[ 
	 \bigl\| \boldsymbol{u}(s)
	 \bigr\|_{\boldsymbol{H}}
	+ |m_0| \bigl( |\Omega |^{1/2}+|\Gamma |^{1/2} \bigr)
	 \Bigr] \nonumber\\
	&\qquad 
	+ C  \bigl\| \boldsymbol{u}'(s)  \bigr\|_{\mathcal{V}_{\sigma,0}^*}
	 \bigl\| \boldsymbol{u}(s)
	 \bigr\|_{\boldsymbol{H}}, \quad \text{for a.a.\ }s \ge \eta.
\end{align*}
 Therefore, using equations \eqref{p4}, \eqref{p5}
and applying estimates \eqref{est0}, \eqref{est32},
we see that there exists a positive constant $M_3'$, depending on
$M_1$, $M_2$, $c_4$, $c_5$, $|\Omega|$, $|\Gamma |$, $L$,  and $m_0$ such that
\begin{align}
	& \bigl| m\bigl( \boldsymbol{\mu }(s)\bigr)\bigr|\nonumber \\
	& \quad \le \frac{1}{|\Omega |+|\Gamma |}
	 \left[ \int_{\Omega }^{} \bigl| \beta \bigl( u(s) \bigr)\bigr| dx
	+  \int_{\Gamma}^{} \bigl| \beta_\Gamma \bigl( u_\Gamma (s) \bigr)\bigr| d\Gamma
	+ \int_{\Omega }^{} \bigl| \pi \bigl( u(s) \bigr)\bigr| dx
	+  \int_{\Gamma}^{} \bigl| \pi_\Gamma \bigl( u_\Gamma (s) \bigr)\bigr| d\Gamma \right]  \nonumber
	\\
	& \quad \le  M_3', \quad \text{for a.a.\ }s \ge \eta.
	\label{esmum}
\end{align}
 Combining this estimate with \eqref{estmu2},
we can apply the generalized {P}oincar\'e inequality given by Lemma~A.3 to achieve the conclusion \eqref{est33}. 
\end{proof}

\begin{remark}\label{regrem}
Thanks to \eqref{est32} and \eqref{est33}, we are able to rewrite \eqref{p3} as 
\begin{equation*}
	\bigl(\boldsymbol{u}'(t), \boldsymbol{z}
	\bigr)_{\boldsymbol{H}}
	+ a_\sigma \bigl( \boldsymbol{\mu }(t), \boldsymbol{z} \bigr) = 0,
	\quad {\it for~all~} \boldsymbol{z} \in \mathcal{V}_{\sigma}\text{\ \ and\ \ } 
	{\it a.a.\ } t \in [\eta,\infty).
\end{equation*}
Regarding this as an elliptic problem for $\boldsymbol{\mu}$, 
then from the elliptic regularity theory (for $\sigma>0$, 
see {\rm Lemma \ref{esLepH2})}, \eqref{est32} and \eqref{est33}, we see that  for any $T >\eta>0$, 
\begin{equation*}
	\boldsymbol{\mu }\in L^2(\eta,T;\mathcal{W}_\sigma).
\end{equation*}
This allows us to obtain the strong form of equation \eqref{p3}:
\begin{align}
	&\partial _t u -\Delta \mu =0, \quad {\rm a.e.\ in~ } (\eta,\infty ) \times \Omega ,
	\label{aein}
	\\
	&\partial _t u_\Gamma +\partial _{\boldsymbol{\nu }} \mu
	-\sigma \Delta _\Gamma \mu _{\Gamma } = 0, \quad {\rm a.e.\ in~ } (\eta,\infty ) \times \Gamma.
	\label{aeon}
\end{align}
See, e.g., \cite[Section~4]{CF15} for related discussions when $\sigma=1$.
\end{remark}

The following lemma is a generalization of \cite[Corollary~4.3]{GGM17} from the case of  homogeneous {N}eumann boundary condition to the current higher-order nonlinear boundary condition \eqref{p5} with singular term.
\begin{lemma} \label{reguH2}
For any $\eta>0$, there exists a positive constant $M_4$ such that
\begin{align}
	&  \bigl \| \beta (u)  \bigr \|_{L^\infty (\eta,\infty;L^p(\Omega))}
	+   \bigl \| \beta (u_\Gamma )  \bigr \|_{L^\infty (\eta,\infty;L^p(\Gamma))} \le M_4,
	\label{est341} \\
	&  \bigl \|\beta_\Gamma(u_{\Gamma})  \bigr \|_{L^\infty(\eta,\infty; H_\Gamma)} 
	\leq M_4,\label{est342}\\
	&\|\boldsymbol{u}\|_{L^\infty (\eta,\infty ;\boldsymbol{W})} \le M_4.
	\label{est34}
\end{align}
In \eqref{est341}, when $n=3$, $p\in[1,6]$ if $\sigma>0$ and $p\in [1,4]$ if $\sigma=0$; 
when $n=2$, $p\in [1,\infty)$. 
\end{lemma}
\begin{proof}
As in \cite{GGM17}, for each $k\in \mathbb{N}\setminus \{1\}$, we define the {L}ipschitz continuous function $h_k: \mathbb{R} \to \mathbb{R}$ by
\begin{equation}
	h_k(r):=
	\begin{cases}
	\displaystyle -1+\frac{1}{k} & \displaystyle {\rm if~} r<-1+\frac{1}{k},\\
	r & \displaystyle {\rm if~}-1+\frac{1}{k} \le r \le 1-\frac{1}{k}, \\
	\displaystyle 1-\frac{1}{k} & \displaystyle {\rm if~}r>1-\frac{1}{k}.
	\end{cases}
	\label{apphk}
\end{equation}
Define 
\begin{align}
\begin{split}
	u_k(s)& :=h_k\circ u(s) \quad \text{ in }V, \\
	u_{\Gamma,k}(s)& :=h_k\circ u_\Gamma(s) \quad \text{ in }V_\Gamma \ \text{ for a.a.\ }s \ge \eta.
\end{split}
\label{uk}
\end{align}
 We have 
$\boldsymbol{u}_k:=(u_k,u_{\Gamma,k})\in C([\eta,\infty);\boldsymbol{V})$ for any $\eta>0$ and 
\begin{equation*}
	\nabla u_k=\nabla u \chi_{[-1+\frac{1}{k},\,1-\frac{1}{k}]}(u),\quad \nabla_\Gamma u_{\Gamma,k}=\nabla_\Gamma u_\Gamma \chi_{[-1+\frac{1}{k},\,1-\frac{1}{k}]}(u_\Gamma),
\end{equation*}
see, e.g., \cite[Corollary~A.6, Chapter~II]{KS00}. 
For any $k\in \mathbb{N}\setminus \{1\}$ and $p\geq 2$, we see that 
$\beta_k:=|\beta(u_k)|^{p-2}\beta(u_k)\in C([\eta,\infty);V)$ is well-defined and 
\begin{equation*}
	\nabla \beta_k=(p-1)  \bigl | \beta(u_k)  \bigr|^{p-2} \beta'(u_k)\nabla u_k.
\end{equation*}
Besides, we note that $ (\beta_k)_{|_\Gamma}=|\beta(u_{\Gamma,k})|^{p-2}\beta(u_{\Gamma,k})\in C([\eta,\infty);V_\Gamma)$. 
Denote 
\begin{equation*}
	\widetilde{\mu}{ {}:={}}\mu-\pi(u),\quad 
	\widetilde{\mu}_\Gamma { {}:={}} \mu_\Gamma-\pi_\Gamma(u_\Gamma).
\end{equation*}
Then 
$\widetilde{\mu}\in L^\infty(\eta,\infty; V)$,  and 
$\widetilde{\mu}_\Gamma \in L^\infty(\eta,\infty; V_\Gamma)$ 
if $\sigma>0$, $\widetilde{\mu}_\Gamma \in L^\infty(\eta,\infty;  H^{1/2}(\Gamma))$ if $\sigma=0$ (see Lemma \ref{muV}).

We start to estimate $\beta(u)$ and $\beta (u_\Gamma)$. For this purpose, multiplying the equation (cf. \eqref{p4})
\begin{equation}
	-\Delta u + \beta (u) = \widetilde{\mu },
	\label{p4a}
\end{equation}
 by $\beta_k$, integrating over $\Omega$, using integration by parts and \eqref{p5} for the term  $\partial_{\boldsymbol{\nu}} u$, we get 
\begin{align}
	& \int_\Omega  \bigl | \beta  \bigl(u_k(s) \bigr)  \bigr|^{p-2}
	\beta \bigl( u_k(s) \bigr) \beta \bigl(u(s) \bigr) dx 
	 + \int_\Gamma \bigl| 
	 \beta  \bigl( 
	 u_{\Gamma,k}(s)  \bigr)
	  \bigr|^{p-2}
	 \beta  \bigl( 
	 u_{\Gamma,k}(s)
	  \bigr) 
	 \beta_\Gamma
	  \bigl( 
	 u_\Gamma(s)
	  \bigr) 
	 d\Gamma \nonumber\\
	& \quad =
	-(p-1)\int_{\Omega }^{} 
	 \bigl|
	\beta
	 \bigl(
	u_k(s)
	 \bigr)
	 \bigr|^{p-2} 
	\beta'
	 \bigl(
	u_k(s)
	 \bigr)
	\nabla u_k(s)\cdot  \nabla u(s) dx \nonumber\\
	&\qquad 
	{} - (p-1) \int_{\Gamma }^{}  
	 \bigl| 
	\beta \bigl( u_{\Gamma,k}(s)  \bigr) 
	 \bigr|^{p-2} 
	\beta' \bigl( u_{\Gamma,k}(s)  \bigr) 
	\nabla_\Gamma u_{\Gamma,k}(s) \cdot  \nabla_\Gamma u_\Gamma(s) d\Gamma\nonumber\\
	&\qquad 
	{}+ \int_\Omega \widetilde{\mu}  (s)
	\bigl | 
	\beta \bigl( u_k(s)  \bigr)
	 \bigr|^{p-2}\beta \bigl(u_k(s) \bigr) dx
	+ \int_\Gamma \widetilde{\mu}_\Gamma  (s)
	 \bigl| 
	\beta \bigl(u_{\Gamma,k}(s) \bigr)  \bigr|^{p-2}
	\beta \bigl(u_{\Gamma,k}(s) \bigr) d\Gamma\nonumber\\
   & \quad { {}=: {} } I_1+I_2+I_3+I_4,
   \label{II}
\end{align}
for a.a.\ $s\geq \eta$. 
From assumption (A1), we infer that $I_1\leq 0$ and $I_2\leq 0$. 
 Next,  by the {S}obolev embedding theorem, the H\"{o}lder and Young inequalities, we obtain 
\begin{equation*}
	I_3 \leq \frac12 \bigl\| 
	\beta(u_k) \bigr\|_{L^p(\Omega)}^p + C
	\| \widetilde{\mu} \|_{L^p(\Omega)}^p
	\leq \frac12  \bigl\| 
	\beta(u_k)  \bigr\|_{L^p(\Omega)}^p+C \|\widetilde{\mu}\|_{V}^p,
\end{equation*}
and 
\begin{align*}
	I_4&\leq \frac12  \bigl \| \beta(u_{\Gamma,k}) 
	 \bigr\|_{L^q(\Gamma)}^q + 
	C \|\widetilde{\mu}_\Gamma\|_{L^q(\Gamma)}^q\nonumber\\
	& \leq  
	\begin{cases}
	 \displaystyle \frac12 \bigl \| \beta(u_{\Gamma,k}) \bigr \|_{L^q(\Gamma)}^q + 
	C \|\widetilde{\mu}_\Gamma\|_{V_\Gamma}^q, & \text{if }\sigma>0, \smallskip \\
	 \displaystyle \frac12 \bigl \|\beta(u_{\Gamma,k}) \bigr \|_{L^q(\Gamma)}^q
	+C \|\widetilde{\mu}_\Gamma\|_{H^{1/2}(\Gamma)}^q, & \text{if }\sigma=0,
	\end{cases}
\end{align*}
where in the above estimates, when $n=3$, $p\in[2,6]$, and $q\in[2,\infty)$ if $\sigma>0$, 
$q\in [2,4]$ if $\sigma=0$; when $n=2$, $p\in[2,\infty)$ and $q\in [2,\infty)$.
 As a remark, throughout this proof, the reader should keep in mind that the meaning of 
 $C$ changes from line to line 
 and even within the same chain of inequalities, whereas those constants are always denoted by $C$.
 Now from assumptions (A1), we see that $\beta(r)$, $\beta_\Gamma(r)$ 
as well as $\beta(h_k(r))$ have the same sign for all $r\in (-1,1)$, 
this fact combined \eqref{apphk} that for any $k\in \mathbb{N}\setminus \{1\}$
\begin{equation}
	\beta(u_k)^2\leq \beta(u_k)\beta(u),
	\quad \beta(u_{\Gamma,k})^2\leq \beta(u_{\Gamma,k})\beta(u_{\Gamma}).
	\label{esabeta}
\end{equation}
 Moreover, it follows from assumption (A2) that 
\begin{alignat}{2}
	& 0 {{} \le {} } \beta(u_{\Gamma,k}) 
	\beta(u_\Gamma) \leq \beta(u_{\Gamma,k}) 
	 \bigl[ 
	c_1\beta_\Gamma(u_\Gamma)+c_2 
	 \bigr],
	&\text{if~} &  u_\Gamma>0,
	\label{beb1}\\
	&0<\beta(u_{\Gamma,k}) \beta(u_\Gamma)\leq -\beta(u_{\Gamma,k})
	 \bigl[ {-c_1\beta_\Gamma(u_\Gamma)+c_2}
	 \bigr], \quad & \text{if~} & u_\Gamma<0.
	\label{beb2}
\end{alignat}
As a consequence, for any  $p,q \geq 2$, we obtain
\begin{equation*}
	\bigl \| \beta(u_k) \bigr \|_{L^p(\Omega)}^p
	\leq \int_\Omega \bigl |
	\beta(u_{k})\bigr|^{p-2} 
	\beta(u_{k}) \beta(u) dx,
\end{equation*}
and 
\begin{align*}
	 \bigl \| \beta(u_{\Gamma,k})  \bigr\|_{L^q(\Gamma)}^q
	& \leq \int_\Gamma  \bigl| 
	\beta(u_{\Gamma,k})  \bigr|^{q-2} 
	\beta(u_{\Gamma,k}) \beta(u_\Gamma) d\Gamma\\
	&\leq c_1 \int_\Gamma  \bigl | \beta(u_{\Gamma,k})  \bigr|^{q-2}\beta(u_{\Gamma,k}) 
	\beta_\Gamma(u_\Gamma) d\Gamma
	+c_2 \int_\Gamma  \bigl | \beta(u_{\Gamma,k})  \bigr|^{q-1}d\Gamma\\
	&\leq c_1 \int_\Gamma  \bigl | \beta(u_{\Gamma,k})  \bigr|^{q-2}\beta(u_{\Gamma,k}) 
	\beta_\Gamma(u_\Gamma) d\Gamma + \frac12  \bigl\| 
	\beta(u_{\Gamma,k})  \bigr\|_{L^q(\Gamma)}^q +C,
\end{align*}
 a.e.\ on $[\eta, \infty)$, 
where $C$ is a positive constant that only depends on $c_2$, $|\Gamma|$, and $q$. 

Combining the above estimates, we deduce from \eqref{II} that 
\begin{align*}
	\bigl \|
	\beta(u_{k}) 
	 \bigr \|_{L^p(\Omega)}^p + 
	 \bigl \| 
	\beta(u_{\Gamma,k}) 
	 \bigr \|_{L^p(\Gamma)}^p
	\leq 
	\begin{cases}
	   C \|\widetilde{\mu}\|_{V}^p+C\|\widetilde{\mu}_\Gamma\|_{V_\Gamma}^p+C, 
	   & \text{if }\sigma>0,\\
	   C \|\widetilde{\mu}\|_{V}^p+C\|\widetilde{\mu}_\Gamma\|_{H^{1/2}(\Gamma)}^p + C, 
	   & \text{if }\sigma=0,
	   \end{cases}
\end{align*}
when $n=3$, $p\in[2,6]$ if $\sigma>0$, $p\in [2,4]$ if $\sigma=0$; 
when $n=2$,  $p\in [2,\infty)$, 
the constant $C$ may depend on $c_1$, $c_2$, $\Omega$, $\Gamma$, $p$ but is independent of $k$. 
Passing to the limit as $k\to\infty$, owing to the Fatou lemma, we conclude from \eqref{est0} and \eqref{est33} that
\begin{equation}
	 \bigl \| 
	\beta  \bigl ( u(s)  \bigr) 
	 \bigr\|_{L^p(\Omega)}
	+ \bigl \| 
	\beta 
	 \bigl ( 
	u_\Gamma (s) 
	 \bigr) 
	 \bigr \|_{L^p(\Gamma)} \le C,\quad \text{ for a.a.\ }s\geq \eta. 
	\label{est341a} 
\end{equation}
The case $p\in[1,2)$ can be easily handled by the H\"{o}lder inequality. 

Next, we estimate the boundary potential $\beta_\Gamma(u_\Gamma)$. 
Multiplying the equation (cf.\ \eqref{p5}) 
\begin{equation}
	\partial _{\boldsymbol{\nu }} u
	- \Delta_\Gamma  u_\Gamma  +
	\beta_\Gamma (u_\Gamma ) = \mu_\Gamma  -\pi_\Gamma (u_\Gamma )=: \widetilde{\mu }_\Gamma
	\quad {\rm a.e.\ on~}\Sigma,\label{p5a}
\end{equation}
by $\beta_\Gamma(u_{\Gamma,k})\in C([\eta,\infty);V_\Gamma)$,
integrating over $\Gamma$, after integration by parts, we get
\begin{align}
	& \int_\Gamma \beta_\Gamma 
	 \bigl ( u_{\Gamma,k}(s) 
	 \bigr) \beta_\Gamma
	 \bigl( 
	u_\Gamma(s)
	 \bigr) d\Gamma \nonumber\\
	& \quad 
	= - \int_{\Gamma }^{} 
	\beta'_\Gamma  \bigl( 
	u_{\Gamma,k}(s) 
	 \bigr) 
	\nabla_\Gamma u_{\Gamma,k}(s)\cdot  \nabla_\Gamma u_\Gamma(s) d\Gamma
	+ \int_\Gamma \widetilde{\mu}_\Gamma (s)\beta_\Gamma 
	 \bigl ( u_{\Gamma,k}(s)  \bigr) d\Gamma\nonumber\\
	&\qquad 
	-\int_\Gamma \partial_{\boldsymbol{\nu}} u(s) \beta_\Gamma 
	 \bigl( u_{\Gamma,k}(s)  \bigr) d\Gamma\nonumber\\
	&\quad { {}=:{}} I_5+I_6+I_7,
	\label{III}
\end{align}
for a.a.\ $s\geq \eta$. 
Similar to $I_1$, we see that $I_5\leq 0$. $I_6$ and $I_7$ can be estimated as follows:
\begin{equation*}
	I_6 \leq \frac14  \bigl \| 
	\beta_\Gamma(u_{\Gamma,k}) 
	 \bigr \|_{H_\Gamma}^2
	+ \| \widetilde{\mu}_\Gamma\|_{H_\Gamma}^2, \quad 
	I_7 \leq \frac14 
	 \bigr \| 
	\beta_\Gamma(u_{\Gamma,k}) 
	 \bigr \|_{H_\Gamma}^2
	+ \|\partial _{\boldsymbol{\nu }} u\|_{H_\Gamma}^2.
\end{equation*}
Besides, by the trace theorem (see e.g., \cite{AF03}), Lemma \ref{inter} and 
 the Young inequality, we see that for some $r\in ( 3/2, 2)$, it holds
\begin{equation*}
	\|
	\partial _{\boldsymbol{\nu }} u 
	\|_{H_\Gamma}^2
	\leq C \|  u \| _{H^r(\Omega)}^2
	\leq \zeta \|  u \|_{H^2(\Omega)}^2
	+ C_\zeta \|  u \|_{H}^2,
\end{equation*}
for any $\zeta>0$. Similar to \eqref{esabeta}, using the fact
\begin{equation*}
	\beta_\Gamma(u_{\Gamma,k})^2\leq \beta_\Gamma(u_{\Gamma,k})\beta_\Gamma(u_{\Gamma}),
\end{equation*}
we deduce from \eqref{III} that 
\begin{equation*}
	 \bigl \| \beta_\Gamma(u_{\Gamma,k}) 
	 \bigr \| _{H_\Gamma}^2 
	\leq 2 \| \widetilde{\mu}_\Gamma\|_{H_\Gamma}^2+ 2\zeta \| u \|_{H^2(\Omega)}^2+ 2C_\zeta\| u \|_{H}^2.
\end{equation*}
Passing to the limit as $k\to\infty$, it follows that 
\begin{equation}
	 \bigl \| \beta_\Gamma(u_{\Gamma})  \bigr \|_{H_\Gamma}^2 
	\leq 2 \|\widetilde{\mu}_\Gamma\|_{H_\Gamma}^2
	+ 2\zeta \| u \|_{H^2(\Omega)}^2+ 2C_\zeta\| u \|_{H}^2.
	\label{bGam}
\end{equation}
From the elliptic regularity theory Lemma \ref{esLepH2}, we have 
\begin{align*}
	& \| \boldsymbol{u}\|_{\boldsymbol{W}}  \\
	& \leq C \left( 
	 \bigl \| \beta(u) \bigr \|_H +  \bigl \|\pi(u) \bigr \|_H
	 {}+ \| \mu \|_H 
	+  \bigl \| \beta_\Gamma(u_{\Gamma}) \bigr \|_{H_\Gamma}
	+  \bigl \| \pi_\Gamma(u_\Gamma) \bigr \|_{H_\Gamma}
	 {}+ \| \mu_\Gamma  \|_{H_\Gamma } 
	+\|\boldsymbol{u}\|_{\boldsymbol{H}}
	\right),
\end{align*}
 a.e.\ on $[\eta, \infty)$. 
In view of \eqref{est0}, \eqref{est33}, \eqref{est341a} and taking the coefficient $\zeta$ sufficiently small in \eqref{bGam}, we get
\begin{equation}
	 \bigl \|
	\boldsymbol{u}(s)
	 \bigr \|_{\boldsymbol{W}} 
	\leq \frac12 \bigl \| \boldsymbol{u}(s) \bigr \|_{\boldsymbol{W}} 
	+ C, \quad \text{for a.a.\ } s\geq \eta, \label{uWa}
\end{equation}
which together with \eqref{est0}, \eqref{est33} and \eqref{bGam} further implies 
\begin{equation}
	 \bigl \| \beta_\Gamma \bigl( u_{\Gamma}(s) \bigl ) \bigr \|_{H_\Gamma} 
	\leq C,\quad \text{for a.a.\ } s\geq \eta.
	\label{bGama}
\end{equation}

Collecting the above estimates \eqref{est341a}, \eqref{uWa} and \eqref{bGama},  after choosing the constant $M_4$ suitably large, we arrive at our conclusion \eqref{est341}--\eqref{est34}.
\end{proof}
In summary, thanks to Lemmas \ref{muV}, \ref{reguH2} and Remark \ref{regrem}, since $\eta>0$ is arbitrary, we see that every global weak solution $(\boldsymbol{u},\boldsymbol{\mu })$ to problem \eqref{GMS} becomes a global strong solution instantaneously when $t>0$.

\section{Separation Property}
\setcounter{equation}{0}

In this section, we prove the separation property of global weak solutions $\boldsymbol{u}(t)$ stated in Theorem \ref{septhm}. 

\subsection{Eventual separation from pure states}
The eventual separation property for sufficiently large time is obtained by a dynamical approach (see e.g., \cite{AW07,GMS09}). 

For any given number $ a \in (-1,1)$, we introduce the phase space (cf.\ (A4))
\begin{equation*}
	\Phi_{ a}:=
	\bigl \{ 
	\boldsymbol{u}=(u,u_\Gamma)\in \boldsymbol{V} 
	{ :} \ \widehat{\beta}(u)\in L^1(\Omega),\ 
	\widehat{\beta}_\Gamma(u_\Gamma)\in L^1(\Gamma), \ 
	m(\boldsymbol{u})={}  a \bigr \}. 
\end{equation*}
Then we have 
\begin{proposition}\label{semigroup}
Assume that the assumptions in Proposition \ref{gloweak} are satisfied. The initial boundary value problem \eqref{GMS} defines a strongly continuous semigroup $S(t):\Phi_{m_0}\mapsto \Phi_{m_0}$ such that 
\begin{equation*}
	S(t)\boldsymbol{u}_0=\boldsymbol{u}(t),\quad \text{for all }t\geq 0,
\end{equation*}
where $\boldsymbol{u}(t)$ is the unique global weak solution to problem \eqref{GMS} subject to the initial datum $\boldsymbol{u}_0\in \Phi_{m_0}$.  
\end{proposition}
\begin{proof}
We infer from \eqref{p1} that $\boldsymbol{u}(t)\in C([0,\infty);\boldsymbol{H})$. Thanks to (A1),  
$\widehat{\beta}$, $\widehat{\beta}_\Gamma$ are proper, convex and lower semi-continuous functionals on 
$H$, $H_\Gamma$, respectively. Hence, from this fact, (A3) and the strong convergence 
$\lim_{t\to 0} \boldsymbol{u}(t)=\boldsymbol{u}_0$ in $\boldsymbol{H}$, we get 
\begin{align*}
	& \lim_{t\to 0}\int_\Omega \widehat{\beta}  \bigl( u(t)  \bigr) dx 
	= \int_\Omega \widehat{\beta}(u_0) dx,\quad 
	\lim_{t\to 0}\int_\Omega \widehat{\pi}  \bigl(u(t)  \bigr) dx
	= \int_\Omega \widehat{\pi}(u_0) dx,\\
	& \lim_{t\to 0}\int_\Gamma \widehat{\beta}_\Gamma  \bigl( u_\Gamma(t)  \bigr) d\Gamma 
	= \int_\Gamma \widehat{\beta}_\Gamma(u_{0\Gamma}) d\Gamma,\quad 
	\lim_{t\to 0}\int_\Gamma \widehat{\pi}_\Gamma  \bigl( u_\Gamma(t)  \bigr) d\Gamma
	= \int_\Gamma \widehat{\pi}_\Gamma(u_{0\Gamma}) d\Gamma.
\end{align*}
On the other hand, recall \eqref{equality}, since $\eta>0$ is arbitrary, we deduce the energy equality  
\begin{equation}
	E\bigl(\boldsymbol{u}(t) \bigr)
	+ \int_{0}^{t}  \bigl \| \boldsymbol{u}'(s)  \bigr \|_{\mathcal{V}_{\sigma,0}^*}^2 ds
	= E\bigl(\boldsymbol{u}(0) \bigr)=E(\boldsymbol{u}_0),\quad \text{for all }t>0.
	\label{einequal}
\end{equation}
Then it holds $\lim_{t\to 0}\|\boldsymbol{u}(t)\|_{\boldsymbol{V}}
=\|\boldsymbol{u}_0\|_{\boldsymbol{V}}$. 
Since $\boldsymbol{u}(t)\in C_w([0,\infty);\boldsymbol{V})$ 
due to  \eqref{p1}, then we obtain the strong convergence 
$\lim_{t\to 0}  \|\boldsymbol{u}(t)-\boldsymbol{u}_0\|_{\boldsymbol{V}}=0$. 
This combined with \eqref{addconti} further implies that 
$\boldsymbol{u}(t)\in C([0,\infty); \Phi_{m_0})$. On the other hand, from  \eqref{unique1}, \eqref{est34} and the interpolation inequality we  infer that $S(t)\in C( \Phi_{m_0}, \Phi_{m_0})$ for all $t\ge 0$ (noting that $S(0)=I$).  
\end{proof}

Next, we consider the stationary problem corresponding to \eqref{GMS}, which can be (formally) obtained by neglecting those time derivatives.
\begin{equation}
\left\{
	\begin{aligned}
	&\Delta \mu_s =0, & \text{ a.e.\ in } \Omega, \\
	&\mu_s = -\Delta u_s + \beta (u_s) + \pi (u_s),  & \text{ a.e.\ in } \Omega, \\
	&u_{s|_\Gamma } =u_{\Gamma s},
	\quad \mu_{s|_\Gamma }=\mu_{\Gamma s}, & \text{ a.e.\ on } \Gamma, \\
	&\partial_{\boldsymbol{\nu }} \mu_s
	-\sigma\Delta_\Gamma  \mu_{\Gamma_s}  =0, & \text{ a.e.\ on } \Gamma, \\
	&\mu_{\Gamma s} = \partial _{\boldsymbol{\nu }} u_s
	- \Delta_\Gamma  u_{\Gamma s}  + \beta_\Gamma (u_{\Gamma s} ) 
	+ \pi_\Gamma (u_{\Gamma s} ), &  \text{ a.e.\ on } \Gamma. 
	\end{aligned}
\right.
	\label{GMSsta}
\end{equation}
It is straightforward to check that, if a pair 
$\boldsymbol{\mu}_s=(\mu_s, \mu_{\Gamma s})$ is a solution to \eqref{GMSsta}, then $\mu_s=\mu_{\Gamma s}$ must be a constant. Thus, system \eqref{GMSsta} simply reduces to a nonlocal elliptic boundary value problem for $\boldsymbol{u}_s=(u_s, u_{\Gamma s})$:
\begin{equation}
\left\{
	\begin{aligned}
	&\mu_s = -\Delta u_s + \beta (u_s) + \pi (u_s),  & \text{ a.e.\ in } \Omega, \\
	&u_{s|_\Gamma } =u_{\Gamma s}, & \text{ a.e.\ on } \Gamma, \\
	&\mu_{s} = \partial _{\boldsymbol{\nu }} u_s
	- \Delta_\Gamma  u_{\Gamma s}  + \beta_\Gamma (u_{\Gamma s} ) 
	+ \pi_\Gamma (u_{\Gamma s} ), & \text{ a.e.\ on } \Gamma, 
	\end{aligned}
\right.
	\label{GMSsta1}
\end{equation}
where 
\begin{equation}
	\mu_s=\frac{1}{|\Omega|+|\Gamma|} \left[ \int_\Omega 
	 \bigl( \beta(u_s) +\pi(u_s) \bigr) dx+
	\int_\Gamma  \bigl(  \beta_\Gamma (u_{\Gamma s} ) 
	+ \pi_\Gamma(u_{\Gamma s}) \bigr) d\Gamma\right].
\label{staweak2}
\end{equation}
More precisely, we introduce 
\begin{definition}\label{weaksta}
A pair $\boldsymbol{u}_s=(u_s,u_{\Gamma s})$ 
is called a steady state of problem \eqref{GMS}, 
if $\boldsymbol{u}_s=(u_s,u_{\Gamma s})\in \Phi_{a}$ 
for some $ a \in (-1,1)$, $\beta(u_s)\in L^2(\Omega)$, 
$\beta_\Gamma(u_{\Gamma s})\in L^2(\Gamma)$ and   
\begin{align}
	& \int_\Omega \nabla  u_s\cdot \nabla z dx+\int_\Omega 
	 \bigl(\beta(u_s)+\pi(u_s)-\mu_s  \bigr) z dx
	+\int_\Gamma \nabla _\Gamma  u_{\Gamma s}\cdot \nabla _\Gamma z_\Gamma d \Gamma \nonumber\\
	&\quad {} + \int_\Gamma \bigl( 
	\beta_\Gamma (u_{\Gamma s})+\pi_\Gamma(u_{\Gamma s})-\mu_s  \bigr) 
	z_\Gamma d\Gamma=0,
	\quad \text{for all }\boldsymbol{z}=(z,z_\Gamma)\in \boldsymbol{V},
\label{staweak1}
\end{align}
with the constant $\mu_s$ given by \eqref{staweak2}. 
\end{definition}
\begin{remark}
The constraint 
$m(\boldsymbol{u}_s)= a$ for some given 
$ a \in (-1,1)$ in Definition \ref{weaksta} is not necessary for the stationary problem. 
It will play a role when we connect problem \eqref{GMSsta1}--\eqref{staweak2} to 
the corresponding evolution problem \eqref{GMS}, due to the mass conservation 
property \eqref{masscov} such that we need to set $ a =m(\boldsymbol{u}_0)=m_0$ (cf. {\rm (A4)}). 
\end{remark}
The following lemma provides a useful characterization on the steady states.
\begin{lemma}\label{solsta}
Assume that  $ a \in (-1,1)$ and {\rm (A1)--(A3)} are satisfied. We denote the set of steady states by 
$\mathcal{S}_{ a}$. 
There exist uniform constants $M_{ a}>0$ and $\delta_{ a} \in (0,1)$ 
such that  every steady state $\boldsymbol{u}_s=(u_s,u_{\Gamma s})\in \mathcal{S}_{ a} $
 and the constant $\mu_s$ satisfy
\begin{alignat}{4}
	-1+\delta_{a} & \leq u_s \leq 1-\delta_{a}, & \quad & \text{ in } \Omega,
	\label{sepsta1}\\
	-1+\delta_{a} & \leq u_{\Gamma s}\leq 1-\delta_{a}, & \quad & \text{ on } \Gamma,
	\label{sepsta2}\\
	|\mu_s | & \leq M_{a}. & & 
\label{bdmus}
\end{alignat}
Moreover, the set $\mathcal{S}_{a}$ is bounded in $H^3(\Omega)\times H^3(\Gamma)$.
\end{lemma}
\begin{proof}
The proof follows from the idea in \cite[Lemma 6.1]{GMS09}. 
Since $\boldsymbol{u}_s=(u_s,u_{\Gamma s})\in \mathcal{S}_{a} $, 
then by (A1) we have $\|u_s\|_{L^\infty(\Omega)}\leq 1$, 
$\|u_{\Gamma s}\|_{L^\infty(\Gamma)}\leq 1$. 
Taking $\boldsymbol{z}=\boldsymbol{u}_s-{a}\boldsymbol{1}$ in \eqref{staweak1}, we have
\begin{align}
	& \int_\Omega  \bigl| \nabla  (u_s-{a})  \bigr|^2 dx
	+\int_\Omega \beta(u_s)(u_s-{ a}) dx
	+\int_\Gamma  \bigl| \nabla _\Gamma  (u_{\Gamma s}-{a}) {\bigr|}^2 d \Gamma \nonumber\\
	&\qquad 
	+ \int_\Gamma \beta_\Gamma (u_{\Gamma s})(u_{\Gamma s}-{a}) d\Gamma\nonumber\\
	&\quad = -\int_\Omega \pi(u_s)(u_s-{a}) dx
	-\int_\Gamma \pi_\Gamma(u_{\Gamma s})(u_{\Gamma s}-{a}) d\Gamma
	\nonumber\\
	&\quad \leq  C,\nonumber
\end{align}
where $C$ 
may depend on $|\Omega|$, $|\Gamma|$, $L$ and ${ a}$. 
From the above estimate and \eqref{gms}, \eqref{gmsa}, \eqref{staweak2}, 
we can easily conclude \eqref{bdmus}. 
We note that $M_{ a}$ is independent of $\boldsymbol{u}_s$. Using (A1) again, 
there exists $\delta_{ a}\in (0,1)$ such that 
\begin{alignat*}{6}
	&\beta(r)+\pi(r)-M_{ a} \geq 1,&  \quad 
	& \beta_\Gamma(r)+\pi_\Gamma(r)-M_{ a} \geq 1,& \quad & 
	\text{for all } r\in[1-\delta_{ a},1],\\
	&\beta(r)+\pi(r)+M_{ a}  \leq -1, & \quad 
	& \beta_\Gamma(r)+\pi_\Gamma(r)+M_{ a} \leq -1,& \quad 
	& \text{for all } r\in[-1,-1+\delta_{ a}],
\end{alignat*}
$\delta_{a}$ is also independent of $\boldsymbol{u}_s$.
Taking the test function in \eqref{staweak1} as 
$\boldsymbol{z}= [\boldsymbol{u}_s-(1-\delta_{a})\boldsymbol{1} ]^+$, 
$\boldsymbol{z}= [\boldsymbol{u}_s+ (1-\delta_{ a})\boldsymbol{1} ]^-$, 
respectively,  where 
$[\boldsymbol{u}]^+:=([u]^+,[u_\Gamma]^+)
:=(\max\{0,u\}, \max\{0,u_\Gamma \})$, 
$ [\boldsymbol{u}]^-:=([u]^-,[u_\Gamma]^-)
:=(-\min\{0,u\}, -\min\{0,u_\Gamma\})$ for $\boldsymbol{u}:=(u,u_\Gamma)$.
 Then, we infer that 
\begin{align*}
	&  \int_\Omega 
	\bigl[ u_s-(1-\delta_a) \bigr]^+  dx
	+\int_\Gamma
	\bigl [ u_{\Gamma s}-(1-\delta_a) \bigr ] ^+
	d \Gamma  \\
	& {} \quad  + \int_\Omega  \bigl | 
	\nabla   \bigl[ u_s-(1-\delta_{ a})  \bigr]^+  \bigr |^2 dx
	+\int_\Gamma \bigl | \nabla _\Gamma  
	 \bigl [ u_{\Gamma s}-(1-\delta_{ a}) \bigr ] ^+
	 \bigr | ^2 d \Gamma \leq 0,\\
	&  \int_\Omega 
	\bigl[ u_s+(1-\delta_a) \bigr]^- dx
	+\int_\Gamma
	\bigl [ u_{\Gamma s}+(1-\delta_a) \bigr ] ^-
	d \Gamma  \\
	& {} \quad  +
	\int_\Omega   \bigl | 
	\nabla   \bigl [ u_s+ (1-\delta_{ a}) \bigr]^-  \bigr | ^2 dx
	+\int_\Gamma  \bigl | \nabla _\Gamma 
	 \bigl [ u_{\Gamma s}+ (1-\delta_{ a}) \bigr]^- \bigr |^2 d \Gamma \leq 0,
\end{align*}
which leads to \eqref{sepsta1}--\eqref{sepsta2}. 
Finally, the separation property and assumptions (A1)--(A2) enable us to apply the elliptic regularity theory (see Lemma \ref{esLepH2}) to conclude that $\boldsymbol{u}_s\in H^3(\Omega)\times H^3(\Gamma)$. 
\end{proof}

Returning to the evolution problem \eqref{GMS}, for any initial datum $\boldsymbol{u}_0 $ satisfying (A4), we define the $\omega $-limit set $\omega (\boldsymbol{u}_0)$
as follows:
\begin{equation*}
	\omega (\boldsymbol{u}_0):=
	\left\{
	\boldsymbol{u}_\infty
	\in [H^{2r}(\Omega)\times H^{2r}(\Gamma)] \cap \Phi_{m_0}:
	\begin{array}{l}
	\mbox{there exists } \{t_n \}_{n \in \mathbb{N}}
	\ \mbox{with } t_n \nearrow \infty \\
	\mbox{such that }
	\boldsymbol{u}(t_n) \to \boldsymbol{u}_\infty \\ {\rm in~}H^{2r}(\Omega)\times H^{2r}(\Gamma) 
	\mbox{ as } n \to \infty
	\end{array}
	\right\},
\end{equation*}
for $r\in[ 1/2,1)$. 

Then we show the relationship between $\omega (\boldsymbol{u}_0)$ and the set of steady states $\mathcal{S}_{m_0}$.

\begin{lemma}\label{ome}
Under the assumptions {\rm (A1)--(A4)},
$\omega (\boldsymbol{u}_0)$ is a non-empty, connected and compact set in 
$H^{2r}(\Omega)\times H^{2r}(\Gamma)$ for $r\in[ 1/2,1)$. Moreover, 
\begin{align}
\omega (\boldsymbol{u}_0)\subset \mathcal{S}_{m_0}\label{omeS}
\end{align}
such that every element
$\boldsymbol{u}_\infty:=(u_\infty ,u_{\Gamma ,\infty }) \in \omega (\boldsymbol{u}_0)$ is a strong solution to the elliptic boundary value problem
\eqref{GMSsta} with the constant $\mu_\infty$ determined by \eqref{staweak2}. 
\end{lemma}
\begin{proof}
From the estimate \eqref{est34}, 
we see that the orbit $\{\boldsymbol{u}(t)\}_{t\geq \eta}$ is
 relatively compact in $H^{2r}(\Omega)\times H^{2r}(\Gamma)$ for 
 any $r\in[ 1/2,1)$. On the other hand, the free energy 
 $E(\boldsymbol{u})$ defined by \eqref{ene} serves as a strict {L}yapunov function for 
 the semigroup $S(t)$ (see \eqref{einequal}). 
 Therefore, the conclusion of the present lemma follows 
 from the well-known results in the dynamical system (see e.g., \cite[Theorem 4.3.3]{henry}) and Lemma \ref{solsta}. We also refer to \cite[Theorem 3.15]{GMS11} for an alternative proof with minor modifications due to assumptions on $\beta$, $\beta_\Gamma$.  
\end{proof}
\smallskip 

Lemmas \ref{solsta}  and \ref{ome} yield the propterty of uniform separation from pure states $\pm 1$ for any element of the $\omega $-limit set $\omega (\boldsymbol{u}_0)$ (see \eqref{sepsta1}, \eqref{sepsta2} and \eqref{omeS}). This essential fact enables us to prove the eventual separation property for global weak solutions to problem \eqref{GMS}.
\medskip 

\noindent 
\textbf{Proof of Theorem \ref{septhm}. Part (1)}. It follows from the definition of $\omega (\boldsymbol{u}_0)$ that 
\begin{equation*}
	\lim_{t\to \infty} \mathrm{dist}  \bigl( 
	S(t) \boldsymbol{u}_0, \omega(\boldsymbol{u}_0) 
	 \bigr) 
	=0\quad \text{in } H^{2r}(\Omega)\times H^{2r}(\Gamma).
\end{equation*}
By the Sobolev embedding theorem, we see that 
$H^{2r}(\Omega) \times H^{2r}(\Gamma) 
\hookrightarrow C(\overline{\Omega})\times C(\Gamma)$ 
when $r\in ( n/4,1)$ ($n=2,3$). Then thanks to Lemmas \ref{solsta}  and \ref{ome}, we can conclude \eqref{esepe} with the choice 
$$\delta_1=\frac{1}{2}\delta_{m_0},$$ 
where the constant $\delta_{m_0}$ is determined as in Lemma \ref{solsta}. \hfill $\Box$

%%%%%%%%%%%%%%%%%%%%%%%%%%%
\subsection{Instantaneous separation from pure states in two dimensional case}
The improved instantaneous separation property can be achieved by some further higher-order estimates for global weak solutions that only depend on an upper bound for the initial energy $E(\boldsymbol{u}_0)$ and on the average of the total mass $m_0$. In this case, the spatial dimension ($n=2$) and the appearance of the surface diffusion ($\sigma>0$) turn out to be crucial due to the Trudinger--Moser inequality (see Lemma \ref{TM}) and the available regularity on $\boldsymbol{\mu}$ (see \eqref{est33}).

\begin{lemma}\label{bedep}
Let $n=2$, $\sigma>0$. 
For any $1\le p<\infty $ and $\eta>0$, there exists a positive constant $C(p,\eta)$ such that
\begin{equation}
	 \bigl \| \beta' (u)  \bigr \|_{L^p(t,t+1;L^p(\Omega ))}
	+  \bigl \| \beta'(u_\Gamma )  \bigr \|_{L^p(t,t+1;L^p(\Gamma))} 
	\le C(p,\eta),\quad \text{for all }t \ge \eta. 
	\label{esbedep}
\end{equation}
\end{lemma}
\begin{proof}
The proof relies on the idea of \cite[Lemma~7.1]{MZ04} for the {C}ahn--{H}illiard 
equation with homogeneous {N}eumann boundary conditions (see, also \cite[Lemma~5.1]{GGM17}). Nevertheless, necessary modifications have to be made in order to handle the current complicated boundary condition. 

For any $k\in \mathbb{N}\setminus \{1\}$ and $K>0$, let $(u_k, u_{\Gamma,k})$ be defined as \eqref{uk}. Because $u_k$ belongs to the bounded interval $[-1+1/k, 1-1/k]$, then we see from (A1) and the assumption \eqref{A5} that 
\begin{align*}
	\nabla \bigl( \beta (u_k) e^{K|\beta (u_k)|}\bigr)
	= \beta '(u_k)  \Bigl( 1+ K \bigl| \beta (u_k)\bigr| \Bigr) 
	e^{K|\beta (u_k)|}\nabla u_k \in C\bigl([\eta,\infty );H \bigr),
\end{align*}
and $\bigl(\beta (u_k) e^{K|\beta (u_k)|}\bigr)_{|_\Gamma}=\beta \bigl( u_{\Gamma, k}(s) \bigr)e^{K|\beta (u_{\Gamma, k}(s))|}$.
Therefore, testing the equation \eqref{p4a} by $\beta (u_k) e^{K|\beta (u_k)|}$, we get
\begin{align}
	&\int_{\Omega }^{} \bigl( \nabla u(s) \cdot \nabla u_{k}(s) \bigr)
	\beta ' \bigl(u_k(s) \bigr) 
	 \Bigl( 1+ K \bigl| \beta \bigl(u_k(s) \bigr)\bigr| \Bigr) 
	e^{K|\beta(u_k(s))|}dx
	\nonumber \\
	&\qquad 
	- \int_{\Gamma }^{} \partial _{\boldsymbol{\nu }} u(s)
	\beta \bigl( u_{\Gamma, k}(s) \bigr)e^{K|\beta (u_{\Gamma, k}(s))|}d\Gamma
	+ \int_{\Omega }^{}
	\beta \bigl( u(s) \bigr) \beta \bigl( u_k(s) \bigr) e^{ K|\beta(u_k(s))|}dx
	\nonumber \\
	&\quad 
	= \int_{\Omega }^{} \widetilde{\mu }(s)\beta \bigl( u_k(s) \bigr) e^{K|\beta(u_k(s))|}dx,\quad \text{for a.a.\ }s\ge \eta. 
	\label{comp1}
\end{align}
Next, we note that 
\begin{align*}
	\nabla_\Gamma  \bigl( \beta (u_{\Gamma, k}) e^{K|\beta (u_{\Gamma, k})|}\bigr)
	& = \beta '(u_{\Gamma, k})  \Bigl(  1+ K \bigl| \beta (u_{\Gamma, k})\bigr| \Bigr) 
	e^{K|\beta (u_{\Gamma, k})|}\nabla_\Gamma  u_{\Gamma, k}
	\\
	& \in C\bigl([\eta,\infty );H_\Gamma \bigr),
\end{align*}
Then, testing \eqref{p5a}
by
$\beta (u_{\Gamma, k}) e^{K|\beta (u_{\Gamma, k})|}$, we get
\begin{align}
	&\int_{\Gamma }^{} \partial _{\boldsymbol{\nu }} u(s)
	\beta \bigl( u_{\Gamma, k}(s)\bigr)e^{K|\beta (u_{\Gamma, k}(s) )|}d\Gamma
	\nonumber \\
	&\quad  +
	\int_{\Gamma }^{} \bigl( \nabla_\Gamma u_\Gamma (s) \cdot \nabla_\Gamma  u_{\Gamma, k}(s) \bigr)
	\beta ' \bigl(u_{\Gamma, k}(s) \bigr)  \Bigl( 1+ K \bigl| \beta \bigl(u_{\Gamma,k}(s) \bigr)\bigr| \Bigr)
	e^{K|\beta(u_{\Gamma, k}(s))|}d\Gamma
	\nonumber \\
	&\quad 
	+ \int_{\Gamma }^{}
	\beta_\Gamma \bigl( u_\Gamma (s) \bigr) \beta \bigl( u_{\Gamma,k }(s) \bigr)
	e^{K|\beta(u_{\Gamma, k}(s))|}d\Gamma\nonumber\\
	&\ \, = \int_{\Gamma}^{} \widetilde{\mu }_\Gamma (s)\beta \bigl( u_{\Gamma, k}(s) \bigr)
	e^{K|\beta(u_{\Gamma, k}(s))|}d\Gamma,\quad \text{for a.a.\ }s \ge \eta.
	\label{comp2}
\end{align}
From (A1), we see that the first term on the left hand side of \eqref{comp1} and the second term on the left
hand side of  \eqref{comp2} are nonnegative. Then adding \eqref{comp1} and \eqref{comp2} together, we infer from \eqref{esabeta}, \eqref{beb1} and \eqref{beb2} that 
\begin{align}
	&\int_{\Omega }^{}
	\beta \bigl( u_k(s) \bigr)^2 e^{K|\beta(u_k(s))|}dx
	+ \frac{1}{c_1}
	\int_{\Gamma }^{}
	\beta \bigl( u_{\Gamma,k }(s) \bigr)^2
	e^{K|\beta(u_{\Gamma, k}(s))|}d\Gamma
	\nonumber \\
	&\quad \le
	 \int_{\Omega }^{}
	 \bigl| \widetilde{\mu }(s) \bigr|  \bigl| \beta \bigl( u_k(s) \bigr)  \bigr| e^{K|\beta(u_k(s))|}dx
	+ \int_{\Gamma}^{} \left(\bigl| \widetilde{\mu }_\Gamma (s) \bigr|+\frac{c_2}{c_1}\right) 
	 \bigl| \beta \bigl( u_{\Gamma, k}(s) \bigr)  \bigr| 
	e^{K|\beta(u_{\Gamma, k}(s))|} d\Gamma\nonumber\\
	&\quad { {}=: {} } J_1+J_2, \quad \text{for a.a.\ }s \ge \eta.
	\label{JJ}
\end{align}
Applying the generalized {Y}oung inequality
given by Lemma~\ref{young}, there exist positive constants $N$ and $M_5$ that may depend on $K$, $c_1$ but is independent of $k$ such that
\begin{equation*}
	\bigl| \widetilde{\mu }(s) \bigr| \beta \bigl( u_k(s) \bigr)
	e^{K|\beta(u_k(s))|}
	\le
	e^{N| \tilde{\mu }(s)|}+\frac{1}{2}\beta \bigl( u_k(s) \bigr)^2
	e^{K|\beta(u_k(s))|}+M_5,
\end{equation*}
and 
\begin{align}
	&\left(\bigl| \widetilde{\mu }_\Gamma (s) \bigr|+\frac{c_2}{c_1}\right) 
	 \bigl| \beta \bigl( u_{\Gamma, k}(s) \bigr)  \bigr|
	e^{K|\beta(u_{\Gamma, k}(s))|}\nonumber\\
	&\quad \leq e^{N \left(| \widetilde{\mu }_\Gamma (s) |+\frac{c_2}{c_1}\right)}+ \frac1{2c_1} \beta \bigl( u_{\Gamma,k }(s) \bigr)^2
	e^{K|\beta(u_{\Gamma, k}(s))|}+M_5,
	\label{esY2}
\end{align}
for a.a.\ $s \ge \eta$. 
Then we can control $J_1$ as follows (see \cite[Lemma~7.1]{MZ04}): in view of the estimate \eqref{est33}, we can employ the {T}rudinger--{M}oser type inequality given by Lemma~\ref{TM} to conclude that there exists a positive constant
$\widetilde{C}_{\rm TM}$ depending on $C_{\rm TM}$ and $N$ such that
\begin{equation*}
	\int_{\Omega }^{}
	e^{N|\tilde{\mu} (s)|} dx
	\le \widetilde{C}_{\rm TM}
	e^{\widetilde{C}_{\rm TM} \|\tilde{\mu }(s)\|_{V}^2}, \quad 
	\text{ for a.a.\ }s \ge \eta.
\end{equation*}
Hence, 
\begin{align}
J_1\leq \widetilde{C}_{\rm TM}
	e^{\widetilde{C}_{\rm TM} \|\tilde{\mu }(s)\|_{V}^2}
	+ \frac12 \int_\Omega\beta \bigl( u_k(s) \bigr)^2
	e^{K|\beta(u_k(s))|} dx+ M_5|\Omega|.
	\label{J1}
\end{align}
Next, we deduce from the embedding 
$V_\Gamma\hookrightarrow C(\Gamma)$ that 
\begin{align*}
\int_\Gamma e^{N\left(| \widetilde{\mu }_\Gamma (s) |+\frac{c_2}{c_1}\right)} d\Gamma 
\le |\Gamma|e^{N\left(\| \widetilde{\mu }_\Gamma (s) \|_{V_\Gamma}+\frac{c_2}{c_1}\right)}, \quad 
	\text{ for a.a.\ }s \ge \eta,
\end{align*}
which together with \eqref{esY2} implies 
\begin{align}
J_2&\leq |\Gamma|e^{N\left(\| \widetilde{\mu }_\Gamma (s) \|_{V_\Gamma}+\frac{c_2}{c_1}\right)}+ \frac{1}{2c_1}
	\int_{\Gamma }^{}
	\beta \bigl( u_{\Gamma,k }(s) \bigr)^2
	e^{K|\beta(u_{\Gamma, k}(s))|}d\Gamma
	+M_5|\Gamma|. 
	\label{J2}
\end{align}
Combining \eqref{JJ}, \eqref{J1} and \eqref{J2}, we arrive at
\begin{align}
&\int_{\Omega }^{}
	\beta \bigl( u_k(s) \bigr)^2 e^{K|\beta(u_k(s))|}dx
	+ \frac{1}{c_1}
	\int_{\Gamma }^{}
	\beta \bigl( u_{\Gamma,k }(s) \bigr)^2
	e^{K|\beta(u_{\Gamma, k}(s))|}d\Gamma
	\nonumber \\
	&\quad \le 2\widetilde{C}_{\rm TM}
	e^{\widetilde{C}_{\rm TM} \|\tilde{\mu }(s)\|_{V}^2}
	+
	2|\Gamma|e^{N\left(\| \widetilde{\mu }_\Gamma (s) \|_{V_\Gamma}
	+\frac{c_2}{c_1}\right)}+ 2M_5  \bigl( |\Omega|+|\Gamma|  \bigr),
	\label{JJJ}
\end{align}
for a.a. $s \ge \eta$.
 
For any fixed $p \ge 1$, we take $K:=pc_0$, where the constant $c_0$ is given in \eqref{A5}. Then we deduce from the assumptions \eqref{A5}, $\sigma>0$ and estimates \eqref{est33}, \eqref{JJJ} that 
\begin{align}
	& \int_{t}^{t+1} \! \! \int_{\Omega }^{}
	\bigl| \beta' \bigl( u_k(s) \bigr) \bigr|^p dx
	ds
	+ \frac{1}{c_1}
	\int_{t}^{t+1} \! \! \int_{\Gamma }^{}
	\bigl| \beta' \bigl( u_{\Gamma,k }(s) \bigr) \bigr|^p d\Gamma
	ds
	\nonumber \\
	& \quad \le
	\int_{t}^{t+1} \! \! \int_{\Omega }^{}
	\bigl( e^{c_0 |\beta ( u_k(s) )|+c_0} \bigr)^p dx
	ds 
	+ \frac{1}{c_1}
	\int_{t}^{t+1} \! \! \int_{\Gamma }^{}
	\bigl( e^{c_0 |\beta ( u_{\Gamma, k}(s) )|+c_0} \bigr)^p d\Gamma
	ds
	\nonumber \\
	& \quad = C \int_{t}^{t+1} \left(
	\int_{\Omega }^{}
	e^{K |\beta ( u_k(s) )|}dx
	+ \frac{1}{c_1}
	\int_{\Gamma }^{}
	e^{K |\beta ( u_{\Gamma, k}(s) )|} d\Gamma \right) ds
	\nonumber \\
	& \quad \le C \int_{t}^{t+1}
	\left(
	e^K |\Omega |
	+ \int_{\Omega }^{} \beta  \bigl( u_k(s)  \bigr)^2
	e^{K |\beta ( u_k(s) )|} dx
	+
	\frac{e^K}{c_1} |\Gamma |
	\right.
	\nonumber \\
	&
	\left. \qquad {}
	+ \frac{1}{c_1}
	\int_{\Gamma }^{} \beta  \bigl( u_{\Gamma, k}(s)  \bigr)^2
	e^{K |\beta ( u_{\Gamma, k}(s) )|} d\Gamma
	\right) ds
	\nonumber \\
	& \quad \le
	2C\widetilde{C}_{\rm TM}
	e^{\widetilde{C}_{\rm TM} M_3^2}
	+
	2C|\Gamma|e^{N\left(M_3+\frac{c_2}{c_1}\right)} 
	+ 2CM_5  \bigl( |\Omega|+|\Gamma|  \bigr) \nonumber\\
	&\qquad +	
	Ce^{K} \bigl(|\Omega |
	+ c_1^{-1}|\Gamma | \bigr), \quad \text{for all }t \ge \eta,
	\label{lions}
\end{align}
where we have used the fact that, if
$r<1$ then $e^{Kr}\le e^K$, and if $r \ge 1$ then
$e^{Kr} \le r^2 e^{Kr}$, and the positive constant $C$ is independent of $k$.
 
Since we already know that
$|u|<1$ a.e.\ in $Q$
and $|u_\Gamma | <1$ a.e.\ on $\Sigma$, then 
$u_k \to u$ a.e.\ in $Q$ and $u_{\Gamma, k} \to u_\Gamma $
a.e.\ on $\Sigma$ as $k \to \infty $, which imply
\begin{align*}
&	\beta '(u_k) \to \beta '(u)
	\qquad {\rm a.e.\ in~} Q, \\
&	\beta '(u_{\Gamma, k}) \to \beta '(u_\Gamma )
	\quad {\rm a.e.\ in~} \Sigma.
\end{align*}
On the other hand, by virtue of the Lions lemma \cite[Lemme~1.3, Chapitre~1]{Lio68} and the uniform estimate \eqref{lions}, we see that 
\begin{align*}
&	\beta '(u_k) \to \beta '(u)
	\qquad {\rm weakly~in~} L^p  \bigl ( t,t+1;L^p(\Omega ) \bigr), \\
&	\beta '(u_{\Gamma, k}) \to \beta '(u_\Gamma )
	\quad {\rm weakly~in~} L^p  \bigl( t,t+1;L^p(\Gamma ) \bigr),
\end{align*}
as $k \to \infty $, for all $t \ge \eta$. Finally, taking $\liminf_{k \to \infty }$ in \eqref{lions} we easily arrive at the conclusion \eqref{esbedep}.
\end{proof}

Lemma \ref{bedep} enables us to obtain some improved estimates on the time derivative of weak solution.

\begin{lemma}\label{impesut}
Let $n=2$, $\sigma>0$. 
 For any $\eta>0$, there exists a positive constant $M_6$ such that
\begin{align}
	&\|\boldsymbol{u}'\|_{L^\infty (2\eta,\infty ;\boldsymbol{H})}\le M_6,\label{est36a}\\
	&
	\int_{t}^{t+1}  \bigl \| \boldsymbol{u}'(s) \bigr\|_{\boldsymbol{W}}^2 ds\le M_6,
	\quad \text{for all } t \ge 2\eta. 
	\label{est36b}
\end{align}
\end{lemma}
\begin{proof}
Taking the difference of \eqref{aein} at $t=s$ and $t=s+h$, we have
\begin{equation}
	\partial _t \bigl( \partial _t^h u(s) \bigr)-\Delta \partial _t^h \mu (s)=0
	\quad {\rm a.e.\ in~} \Omega
	\label{diaein}
\end{equation}
for a.a.\ $s \ge \eta$. 
Analogously, we obtain from \eqref{aeon} that
\begin{equation}
	\partial _t \bigl( \partial _t^h u_\Gamma (s) \bigr)
	+\partial _{\boldsymbol{\nu }} \partial _t^h \mu (s)
	-\sigma\Delta_\Gamma  \partial _t^h \mu _\Gamma (s)=0
	\quad {\rm a.e.\ on~} \Gamma
	\label{diaeon}
\end{equation}
for a.a.\ $s \ge \eta$.
Multiplying  \eqref{diaein} by $\partial _t^h u(s)$, integrating over $\Omega$,  we get
\begin{align}
	& \frac{1}{2} \frac{d}{ds}  \bigl \| \partial _t^h u(s)  \bigr \|_H^2
	- \int_{\Omega }^{}
	 \partial _t^h \mu (s) \Delta \partial _t^h u(s) dx
	- \int_{\Gamma }^{} \partial _{\boldsymbol{\nu }}
	\partial _t^h \mu (s)
	\partial _t^h u_\Gamma (s) d\Gamma \nonumber\\
	&\quad +\int_\Gamma \partial_t^h \mu_\Gamma(s)  \partial _{\boldsymbol{\nu }}\partial _t^h u (s) d\Gamma =0.
	\label{ad1}
\end{align}
Next, multiplying  \eqref{diaeon} by $\partial _t^h u_\Gamma (s)$, integrating over $\Omega$,  we get
\begin{align}
	\frac{1}{2} \frac{d}{ds}  \bigl \| \partial _t^h u_\Gamma (s)  \bigr \|_{H_\Gamma }^2
	+ \int_{\Gamma }^{} \partial _{\boldsymbol{\nu }}
	\partial _t^h \mu (s)
	\partial _t^h u_\Gamma (s) d\Gamma
	+ \sigma \int_{\Gamma }^{}
	\nabla_\Gamma  \partial _t^h \mu_\Gamma  (s)
	\cdot \nabla_\Gamma  \partial _t^h u_\Gamma (s) d\Gamma
	=0.
	\label{ad2}
\end{align}
On the other hand, taking differences of equations \eqref{p4} and \eqref{p5} at $t=s$ and $t=s+h$, respectively,
we have
\begin{align}
	&\partial _t^h \mu (s)= - \Delta \partial _t^h u(s)
	+\partial _t^h \beta \bigl( u(s)\bigr)+ \partial _t^h \pi \bigl( u(s)\bigr)
	\quad {\rm a.e.\ in~} \Omega,
	\label{di1}\\
	&\partial _t^h \mu _{\Gamma }(s)
	= \partial _{\boldsymbol{\nu }} \partial _t^h u(s)
	- \Delta _\Gamma \partial _t^h u_\Gamma (s)
	+ \partial _t^h \beta_\Gamma \bigl( u_\Gamma (s)\bigr) +\partial_t^h \pi_\Gamma \bigl( u_\Gamma (s)\bigr)
	\quad {\rm a.e.\ on~} \Gamma,
	\label{di2}
\end{align}
for a.a.\ $s \ge \eta$. 
Multiplying  \eqref{di2} by  $\sigma\partial _t^h \mu_\Gamma  (s)$, integrating over $\Gamma$, we get 
\begin{align*}
	\sigma  \bigl \| \partial _t^h \mu _\Gamma (s) \bigr \|_{H_\Gamma }^2  
	&= \sigma \int_{\Gamma}^{}
	\partial _{\boldsymbol{\nu }} \partial _t^h u(s) \partial _t^h \mu _\Gamma (s) d\Gamma
	+ \sigma \int_{\Gamma }^{}
	 \nabla_\Gamma  \partial _t^h u_\Gamma (s) \cdot \nabla_\Gamma  \partial _t^h \mu_\Gamma  (s)
	 d\Gamma
	\nonumber \\
	&\quad 
	+ \sigma \int_{\Gamma }^{}\partial _t^h \bigl( \beta_\Gamma \bigl( u_\Gamma (s)\bigr) + \pi_\Gamma \bigl( u_\Gamma (s)\bigr)\bigr)	\partial _t^h \mu_\Gamma  (s) d\Gamma,
\end{align*}
for a.a.\ $s \ge \eta$.
Adding this with \eqref{ad1} and \eqref{ad2}, using \eqref{di1} and Young's inequality, we obtain that 
\begin{align}
	& \frac{1}{2} \frac{d}{ds}  \Bigr( \bigl \| \partial _t^h u(s)  \bigr \|_H^2
	+   \bigl \| \partial _t^h u_\Gamma (s)  \bigr \|_{H_\Gamma }^2  \Bigr )
 	+  \bigl \| \partial _t^h \Delta u (s)  \bigr \|_H^2
	+  \sigma  \bigl \| \partial _t^h \mu _\Gamma (s)  \bigr \|_{H_\Gamma }^2
	\nonumber \\
	& \quad = (\sigma-1) \int_{\Gamma}^{}
	\partial _{\boldsymbol{\nu }} \partial _t^h u(s) \partial _t^h \mu _\Gamma (s) d\Gamma + \int_\Omega \Delta \partial _t^h u(s)
	\left(\partial _t^h \beta \bigl( u(s)\bigr)+ \partial _t^h \pi \bigl( u(s)\bigr)\right) dx 
	\nonumber\\
	& \qquad {} + \sigma \int_{\Gamma }^{}\partial _t^h \bigl( \beta_\Gamma \bigl( u_\Gamma (s)\bigr) + \pi_\Gamma \bigl( u_\Gamma (s)\bigr)\bigr)	\partial _t^h \mu_\Gamma  (s) d\Gamma\nonumber\\
	& \quad  \leq 	\frac{1}{2}  \bigl \| \partial _t^h \Delta u (s)  \bigr \|_H^2
	+  \frac{\sigma}{2}  \bigl \| \partial _t^h \mu _\Gamma (s)  \bigr \|_{H_\Gamma }^2 
	+ \frac{3(\sigma-1)^2}{2\sigma} \bigl \| 
	\partial _{\boldsymbol{\nu }} \partial _t^h u(s)  \bigr \|_{H_\Gamma}^2
	+   \bigl \| \partial _t^h \beta \bigl( u(s)\bigr)  \bigr \|_H^2
	\nonumber\\
	& \qquad {}+  \bigl\| \partial _t^h \pi \bigl( u(s)\bigr)  \bigr \|_H^2
	+ \frac{3\sigma}{2}  \bigl\|  \partial _t^h \beta_\Gamma \bigl( u_\Gamma (s)\bigr)  \bigr \|_{H_\Gamma }^2
	+ \frac{3\sigma}{2}  \bigl\|  \partial _t^h \pi_{ \Gamma} \bigl( u_\Gamma (s)\bigr)  \bigr \|_{H_\Gamma }^2,
	\label{0}
\end{align}
for a.a.\ $s \ge \eta$. Besides, we see from \eqref{di2} that
\begin{align}
	  \bigl\|  \partial _t^h \mu _\Gamma (s)  \bigr \|_{H_\Gamma }^2
	&\geq \frac{1}{3}   \bigl\| \partial _{\boldsymbol{\nu }} \partial _t^h u(s)
	- \Delta _\Gamma \partial _t^h u_\Gamma (s)  \bigr \|_{H_\Gamma }^2
	-  \bigl \| \partial _t^h \beta_\Gamma \bigl( u_\Gamma (s)\bigr)  \bigr \|_{H_\Gamma }^2\nonumber\\
	& \quad 
	-  \bigl \| \partial_t^h \pi_\Gamma \bigl( u_\Gamma (s)\bigr)  \bigr \|_{H_\Gamma }^2.\label{0a}
\end{align}
Now we proceed to estimate the right hand side of \eqref{0}. First, from the {L}ipschitz continuity of $\pi $, we have
\begin{equation}
	 \bigl \| \partial _t^h \pi \bigl( u (s)\bigr)  \bigr \|_{H}^2
	\le L^2 \bigl \| \partial _t^h  u (s)  \bigr \| _{H}^2,
	\quad  \bigl \| \partial _t^h \pi_\Gamma \bigl( u_\Gamma  (s)\bigr) \bigr \|_{H_\Gamma }^2
	\le L^2  \bigl \| \partial _t^h  u_\Gamma  (s)  \bigr \|_{H_\Gamma }^2.
	\label{1}
\end{equation}
Next, from the convexity of $\beta '$  (see (A1)), we get
\begin{align}
	 \bigl \| \partial _t^h \beta \bigl( u(s)\bigr)  \bigr \|_H^2
	& = \int_{\Omega }^{} \left| \int_{0}^{1}
	\beta '\bigl( \tau u(s+h)+(1-\tau ) u(s) \bigr)\, \partial _{t}^h u(s) 
	d\tau \right|^2 dx
	\nonumber \\
	& \le \left\|
	\int_{0}^{1}
	\left(  \tau
	\beta '\bigl( u(s+h) \bigr) + (1-\tau )
	\beta ' \bigl( u(s) \bigr)
	\right)
	d\tau \right\|_{H}^2
	 \bigl \| \partial _{t}^h u(s)  \bigr \|_{L^\infty(\Omega )}^2
	\nonumber \\
	& \le \left[
	\int_{0}^{1}
	 \Bigl(  \tau  \bigl \|
	\beta '\bigl( u(s+h) \bigr)  \bigr \|_{H} + (1-\tau )
	 \bigl \| \beta ' \bigl( u(s) \bigr)  \bigr \|_{H}
	 \Bigr)
	d\tau \right] ^2
	 \bigl \| \partial _{t}^h u(s)  \bigr \|_{L^\infty(\Omega )}^2
	\nonumber \\
	& \le \left(  \frac{1}{2}
	 \bigl \| \beta '\bigl( u(s+h) \bigr) \bigr \|_{H}
	+ \frac{1}{2}
	\bigl \|\beta ' \bigl( u(s) \bigr)  \bigr \|_{H}  \right) ^2
	 \bigl \| \partial _{t}^h u(s)  \bigr \|_{L^\infty(\Omega )}^2
	\nonumber \\
	& \le \Bigl(
	 \bigl \| \beta '\bigl( u(s+h) \bigr) \bigr \|_{H}^2
	+
	 \bigl \|  \beta ' \bigl( u(s) \bigr)  \bigr \|_{H}^2 \Bigr)
	 \bigl \|  \partial _{t}^h u(s)  \bigr \|_{L^\infty(\Omega )}^2.
	\label{2}
\end{align}
Analogously,
\begin{equation}
	 \bigl \| \partial _t^h \beta_\Gamma \bigl( u_\Gamma (s)\bigr)  \bigr \|_{H_\Gamma }^2
	\le \Bigl(
	 \bigl \| \beta '_\Gamma\bigl( u_\Gamma (s+h) \bigr) \bigr \|_{H_\Gamma}^2
	+
	 \bigl \| \beta '_\Gamma \bigl( u_\Gamma (s) \bigr)  \bigr \|_{H_\Gamma}^2 \Bigr)
	 \bigl \| \partial _{t}^h u_\Gamma (s)  \bigr \|_{L^\infty(\Gamma )}^2.
	\label{3}
\end{equation}
Finally, by the trace theorem and Lemma \ref{inter}, we see that for some $r\in ( 3/2, 2)$, it holds
\begin{align}
	 \bigl \|\partial _{\boldsymbol{\nu }} \partial _t^h u(s)  \bigr \|_{H_\Gamma}^2
	&\leq C  \bigl \| \partial _t^h u(s)  \bigr \|_{H^r(\Omega)}^2\leq 
	\zeta  \bigl \| \partial _t^hu(s)  \bigr \|_{H^2(\Omega)}^2
	+ C_\zeta  \bigl \| \partial _t^hu(s)  \bigr \|_{H}^2,
	\label{4}
\end{align}
for any $\zeta>0$. 

For each
$\boldsymbol{z} \in \boldsymbol{W} \cap \boldsymbol{V}_0$,
it holds
\begin{equation*}
	\bigl( \partial \varphi (\boldsymbol{z}), \boldsymbol{z} \bigr)_{\boldsymbol{H}_0}
	= a_{{1}}(\boldsymbol{z},\boldsymbol{z})
	= \|\boldsymbol{z}\|_{\boldsymbol{V}_0}^2 {}:= \| \boldsymbol{z} \|_{{\mathcal V}_{1,0}}^2.
\end{equation*}
From the generalized {P}oincar\'e inequality given by Lemma~\ref{Poinc} and the elliptic regularity theory (e.g., Lemma \ref{esLepH2}), we have 
\begin{align}
	\|\boldsymbol{z}\|_{\boldsymbol{W}}\leq 
	C \bigl \| \partial \varphi (\boldsymbol{z}) \bigr \|_{\boldsymbol{H}_0}, 
	\quad \text{for all }\boldsymbol{z} \in \boldsymbol{W} \cap \boldsymbol{V}_0.
	\label{5}
\end{align}
Thus, by the two dimensional Agmon inequality (see \cite[Chapter II, (1.40)]{Tem97}) we infer that for any $\boldsymbol{z} \in \boldsymbol{W} \cap \boldsymbol{V}_0$,
\begin{align}
	&\|z\|_{L^\infty(\Omega )}^2 \le C\|z\|_{W}\|z\|_H
	\le  \bigl \| \partial \varphi (\boldsymbol{z}) \bigr \|_{\boldsymbol{H}_0}
	\| \boldsymbol{z} \|_{\boldsymbol{H}_0},
	\label{6}
\end{align}
and by the Sobolev embedding theorem 
\begin{align}
	&\|z_\Gamma \|_{L^\infty(\Gamma  )}^2 \le C\|z_\Gamma\|_{V_\Gamma}^2\le C\|\boldsymbol{z}\|_{\boldsymbol{V}}^2
	\le
	C \bigl \| \partial \varphi (\boldsymbol{z})  \bigr \|_{\boldsymbol{H}_0}
	\| \boldsymbol{z} \|_{\boldsymbol{H}_0}.
	\label{7}
\end{align}
Now combining \eqref{0}--\eqref{7} and taking the constant $\zeta>0$ to be sufficiently small, we obtain that 
\begin{align}
	& \frac{d}{ds}  \bigl \| \partial _t^h \boldsymbol{u}(s)  \bigr \|_{\boldsymbol{H}_0}^2
	+ \min\left\{\frac{1}{2},\ \frac{\sigma}{6}\right\} \bigl \|
	\partial \varphi
	\bigl( \partial _t^h \boldsymbol{u}(s)
	\bigr)
	\bigr \|_{\boldsymbol{H}_0}^2
	\nonumber \\
	&\quad  \le C \mathcal{B}(s)
	 \bigl \| \partial \varphi \bigl(
	\partial _t^h \boldsymbol{u}(s) \bigr) \bigr \|_{\boldsymbol{H}_0}
	 \bigl \|\partial _t^h \boldsymbol{u}(s)  \bigr \|_{\boldsymbol{H}_0}
	+ C  \bigl \| \partial _t^h \boldsymbol{u}(s)  \bigr \|_{\boldsymbol{H}_0}^2
	\nonumber \\
	&\quad  \le
	\min\left\{\frac{1}{4},\ \frac{\sigma}{12}\right\}
	 \bigl \| \partial \varphi \bigl(
	\partial _t^h \boldsymbol{u}(s) \bigr)  \bigr \|_{\boldsymbol{H}_0}^2
	+ C  \bigl( \mathcal{B}(s)^2+1  \bigr)
	 \bigl \| \partial _t^h \boldsymbol{u}(s)  \bigr \|_{\boldsymbol{H}_0}^2,
	\label{8}
\end{align}
for a.a.\ $s \ge \eta$, where
\begin{align*}
	\mathcal{B}(s) 
	& :=	
	 \bigl \| \beta '\bigl( u(s+h) \bigr)\bigr \|_{H}^2
	+
	 \bigl \| \beta ' \bigl( u(s) \bigr)  \bigr \|_{H}^2
	+
	 \bigl \| \beta_\Gamma '\bigl( u_\Gamma (s+h) \bigr) \bigr \|_{H_\Gamma}^2\nonumber\\
	& \quad 
	+
	 \bigl \| \beta_\Gamma ' \bigl( u_\Gamma (s) \bigr)  \bigr \|_{H_\Gamma}^2.
\end{align*}
Thanks to \eqref{est32} and Lemma \ref{bedep}, 
 for all $\eta_1>0$ 
there exists a positive constant 
 $\widetilde{C}(4,\eta_1)$
depending on $M_2$ and  $C(4,\eta_1)$ such that
\begin{equation*}
	\int_{t}^{t+\eta_1}  \bigl \| \partial _t^{h}
	\boldsymbol{u}(s)  \bigr \|_{\boldsymbol{H}_0}^2 ds \le \widetilde{C}(4,\eta_1),
	\quad \int_{t}^{t+\eta_1}\mathcal{B}(s)^2 ds \le \widetilde{C}(4,\eta_1),
\end{equation*}
for all $t \ge \eta$, indeed, we can apply the same way of the proof of Lemma 3.3. Hence applying the uniform {G}ronwall inequality given by Lemma~A.2, we
deduce that
\begin{align}
	 \bigl \| \partial _t^h \boldsymbol{u}(t+\eta_1)  \bigr \|_{\boldsymbol{H}_0}^2
	&  
	 \le \left( \frac{1}{\eta_1} \int_t^{t+\eta_1} 
	 \bigl \| \partial_t ^h \boldsymbol{u}(s) \bigr \|_{\boldsymbol{H}_0}^2 ds \right)
        e^{\int_t^{t+\eta_1}C( {\mathcal B}(s)^2+1) ds  } 
	\nonumber \\
	& \le  \frac{1}{\eta_1} \widetilde{C}(4,\eta_1)
	e^{C(\widetilde{C}(4,\eta_1)+\eta_1)},
	\quad \text{for all }t \ge \eta.
	\label{9}
\end{align}
 Moreover, integrating \eqref{8} with respect to time and using \eqref{9}, we have 
\begin{align}
	&\min\left\{\frac{1}{4},\ \frac{\sigma}{12}\right\}\int_{t}^{t+1} \bigl \|
	\partial \varphi
	\bigl( \partial _t^h \boldsymbol{u}(s)
	\bigr)
	 \bigr \|_{\boldsymbol{H}_0}^2ds\nonumber\\
	&\quad  \le  \bigl \| \partial _t^h \boldsymbol{u}(t)  \bigr \|_{\boldsymbol{H}_0}^2 
	+ C\int_{t}^{t+1}  \bigl( \mathcal{B}(s)^2+1 \bigr)
	 \bigl \| \partial _t^h \boldsymbol{u}(s)  \bigr \|_{\boldsymbol{H}_0}^2ds
	\nonumber \\
	& \quad \le  \frac{1}{\eta_1} \widetilde{C}(4,\eta_1)
	e^{C(\widetilde{C}(4,\eta_1)+\eta_1)}
	\bigl[ 1+C \bigl( \widetilde{C}(4,1)+1 \bigr) \bigr].
	\label{10}
\end{align}
for all $t \ge \eta+\eta_1$.

For simplicity, we just take $\eta_1=\eta$. Letting $h \to 0$ in \eqref{9}, we obtain the estimate \eqref{est36a}.
Moreover, taking $h\to 0$ in \eqref{10}, we conclude from \eqref{5} the second estimate \eqref{est36b}. The proof is complete.
\end{proof}
\smallskip

We are now in a position to finish the proof of Theorem \ref{septhm}.

\smallskip
\textbf{Proof of Theorem \ref{septhm}. Part (2)}. 
Consider \eqref{aein}--\eqref{aeon} as an elliptic problem for $\boldsymbol{\mu}$. Recalling that now we assume $\sigma>0$, then by a similar reasoning for \eqref{5}, we infer from  \eqref{est36a} that 
\begin{align*}
	\bigl \| \boldsymbol{\mu }- m(\boldsymbol{\mu })  \bigr \|_{L^\infty (2\eta,\infty ;\boldsymbol{W})}\leq C\|\boldsymbol{u}'\|_{L^\infty (2\eta,\infty ;\boldsymbol{H}_0)}\leq C M_6.
\end{align*} 
This together with \eqref{esmum} yields 
\begin{align}
\|\boldsymbol{\mu }\|_{L^\infty (2\eta,\infty ;\boldsymbol{W})}\leq \widetilde{M}_7.
\label{muesH2b}
\end{align}
Put $\widetilde{\mu }:=\mu -\pi (u)$ and 
$\widetilde{\mu }_\Gamma :=\mu_\Gamma  -\pi_{\Gamma} (u_\Gamma )$. 
From \eqref{est0}, \eqref{muesH2b}, (A3) and the Sobolev embedding theorem, we see that there exists a positive constant $M_7$ such that
\begin{equation}
	\| \widetilde{\mu} \|_{L^\infty ((t,t+1) \times \Omega )}
	\le M_7,
	\quad \| \widetilde{\mu}_\Gamma \|_{L^\infty ((t,t+1) \times \Gamma  )}
	\le M_7,
	\label{muinfty}
\end{equation}
for all $t \ge 2\eta$.

Thanks to \eqref{muinfty}, we are able to obtain further estimates for
$\beta (u)$ and $\beta (u_\Gamma )$.
This is an essence of the proof for the separation property. 
To this end, for each $p\in[2,\infty) $, testing the equation 
\begin{equation*}
	-\Delta u + \beta (u) =\widetilde{\mu }\in L^\infty \bigl( (t,t+1) \times \Omega  \bigr)
\end{equation*}
by $|\beta (u)|^{p-2}\beta (u)\in L^1((t,t+1) \times \Omega )$
(recall \eqref{est341}), and testing the equation
\begin{equation*}
	\partial _{\boldsymbol{\nu }}u-\Delta_\Gamma  u_\Gamma  + \beta_\Gamma (u_\Gamma )
	=\mu_\Gamma  -\pi (u_\Gamma ) =:\widetilde{\mu }_\Gamma  \in L^\infty \bigl( (t,t+1) \times \Gamma \bigr)
\end{equation*}
by $|\beta (u_\Gamma )|^{p-2}\beta (u_\Gamma )$, adding the resultants together and by a similar argument like in Lemma \ref{reguH2}, we obtain
\begin{align*}
	&
	\int_{\Omega }^{}\bigl| \beta \bigl(u(s) \bigr) \bigr|^{p}dx
	+
	\int_{\Gamma  }^{} \bigl| \beta \bigl(u_\Gamma (s) \bigr) \bigr|^{p-2}\beta \bigl(u_\Gamma (s) \bigr)\beta_\Gamma \bigl(u_\Gamma (s) \bigr) d\Gamma
	\nonumber \\
	& \quad
	\le  \int_{\Omega }^{}\widetilde{\mu }(s)
	 \bigl| \beta \bigl(u(s) \bigr) \bigr|^{p-2} \beta \bigl(u(s) \bigr) dx
	 +\int_{\Omega }^{}\widetilde{\mu }_\Gamma (s)
	 \bigl| \beta \bigl(u_\Gamma (s) \bigr) \bigr|^{p-2} \beta \bigl(u_\Gamma (s) \bigr) d\Gamma\\
	  & \quad \le
	  \bigl \| \widetilde{\mu }(s)  \bigr \|_{L^p(\Omega )}
	 \left(
	 \int_{\Omega }^{}\bigl| \beta \bigl(u(s) \bigr) \bigr|^{p}dx
	\right)^{\frac{p-1}{p}}
	+  \bigl \| \tilde{\mu }_\Gamma(s) \bigr \|_{L^p(\Gamma )}
	 \left(
		\int_{\Gamma  }^{}\bigl| \beta \bigl(u_\Gamma (s) \bigr) \bigr|^{p}d\Gamma
	\right) ^{\frac{p-1}{p}}
\end{align*}
and
\begin{align*}
	\int_\Gamma  \bigl |\beta  \bigl( u_{\Gamma} (s) \bigr) \bigr |^p
	d\Gamma 
	& \leq 2 c_1 \int_\Gamma  \bigl | 
	\beta  \bigl( u_{\Gamma} (s) \bigr) \bigr |^{ p-2}
	\beta \bigr( u_{\Gamma}  (s) \bigr) \beta_\Gamma \bigr( 
	u_\Gamma  (s) \bigr) d\Gamma  +C, \quad \text{for a.a.\ }s \ge 2\eta,
\end{align*}
It follows from the Young inequality that 	 
\begin{align*}
	&\int_{\Omega }^{}\bigl| \beta \bigl(u(s) \bigr) \bigr|^{p}dx
	+ \frac{1}{2c_1}
	\int_{\Gamma  }^{}\bigl| \beta \bigl(u_\Gamma (s) \bigr) \bigr|^{p}d\Gamma\nonumber\\
	 &\quad  \le \frac12\int_{\Omega }^{}\bigl| \beta \bigl(u(s) \bigr) \bigr|^{p}dx
	+ \frac{1}{4c_1}
	\int_{\Gamma  }^{}\bigl| \beta \bigl(u_\Gamma (s) \bigr) \bigr|^{p}d\Gamma
	+ \frac{1}{p}\left(\frac{2(p-1)}{p }\right) ^{p-1}  
	 \bigl \| \widetilde{\mu }(s) \bigr \|_{L^p(\Omega )}^p\nonumber\\
	&\qquad +\frac{1}{2c_1 p}\left(\frac{2(p-1)}{p }\right) ^{p-1}
	  \bigl \| \widetilde{\mu }_\Gamma(s)  \bigr \|_{L^p(\Gamma )}^p +C,
	 \quad \text{for a.a.\ }s \ge 2\eta,
\end{align*}
where $C$ is independent of $p$.  We note that 
$$
\|\widetilde{\mu }\|_{L^p((t,t+1) \times \Omega )}
\to \|\widetilde{\mu }\|_{L^\infty ((t,t+1) \times \Omega )},\quad \|\widetilde{\mu }_\Gamma\|_{L^p((t,t+1) \times \Gamma )}
\to \|\widetilde{\mu }_\Gamma\|_{L^\infty ((t,t+1) \times \Gamma )}
$$
as $p \to \infty$ (see, e.g., \cite[Theorem~2.14]{AF03}). 
Therefore, for sufficiently large $p$, it holds
\begin{align*}
	&\int_{t}^{t+1}
	\! \!
	\int_{\Omega }^{}
	\bigl| \beta \bigl( u(s) \bigr)\bigr|^p dx ds 
	+ \frac{1}{2c_1} \int_{t}^{t+1}
	\! \!\int_{\Gamma  }^{}\bigl| \beta \bigl(u_\Gamma (s) \bigr) \bigr|^{p}d\Gamma ds\\
		& \quad \leq \left(1+\frac{1}{2c_1 p}\right)\left(\frac{2(p-1)}{p }\right) ^{p-1} (M_7)^{p} +C,
		\quad \text{for all }t \ge 2\eta.
\end{align*}
  Letting $p \to \infty $, we deduce that
\begin{equation*}
	 \bigl \| \beta (u)  \bigr \|_{L^\infty ((t,t+1)\times \Omega )}
	\le 2M_7+1, 
	\quad  \bigl \| \beta (u_\Gamma )  \bigr \|_{L^\infty  ((t,t+1)\times \Gamma )}
	\le 2M_7+1.
\end{equation*}
for all $t \ge 2\eta$.
From assumption (A1) on $\beta$, we see that there exists a constant $\delta_2 \in(0,1)$ such that
\begin{equation*}
	 \bigl \| u(t) \bigr \|_{L^\infty ((t,t+1)\times \Omega )} \le 1-\delta_2, \quad
	\bigl \| u_\Gamma (t) \bigr \|_{L^\infty ((t,t+1)\times \Omega )} \le 1-\delta_2,
\end{equation*}
for all $t \ge 2\eta$. Besides, it holds that 
$\boldsymbol{u}\in L^\infty(\eta,\infty; C(\overline{\Omega })\times C(\Gamma))$, due to the {S}obolev embedding theorem and \eqref{est34}. As a consequence, we can conclude the separation property \eqref{isepe} (replacing $2\eta$ by $\eta$ since $\eta>0$ is arbitrary).

The proof of Theorem \ref{septhm} is complete. \hfill $\Box$

%%%%%%%%%%%%%%%%%%%%%%%%%%%%%%%%%%%
\section{Long-time behavior}
\setcounter{equation}{0}

In this section we prove Theorem \ref{convthm} on the long time behavior of problem \eqref{GMS}.

\subsection{Compactness of the orbit}

The following lemma implies the compactness of the weak solution $\boldsymbol{u}$ in $\boldsymbol{W}$ for large time. 
\begin{lemma}\label{comp}
Under the assumptions of Proposition \ref{gloweak}, there exists a positive constant $M_8$ such that
\begin{equation}
	 \|u\|_{L^\infty(T_1,\infty; H^3(\Omega))}
	 +  \|u_\Gamma\|_{L^\infty(T_1,\infty; H^3(\Gamma))}\le M_8,\quad \text{for all } t \ge T_1,
	\label{est42}
\end{equation}
 where $T_1$ is the same as in Theorem \ref{septhm} (1). 
\end{lemma}
\begin{proof}
Consider the equation for $\boldsymbol{u}=(u,u_\Gamma )$ 
\begin{align*}
	-\Delta u(t) &= \mu(t) - \beta \bigl( u(t) \bigr) - \pi \bigl( u(t) \bigr) =:f(t)
	\quad {\rm a.e.\ in~} \Omega ,
	\\
	u_{|_\Gamma }(t)&=u_\Gamma(t)
	\quad {\rm a.e.\ on~} \Gamma,
	\\
	\partial _{\boldsymbol{\nu }} u(t) & - \Delta_\Gamma  u_\Gamma (t) +u_\Gamma(t)=\mu _\Gamma(t)
	- \beta_{\Gamma} \bigl(u_\Gamma(t) \bigr) + \pi_{\Gamma} \bigl( u_\Gamma(t) \bigr) +u_\Gamma(t)
	=: f_\Gamma (t)
	\quad {\rm a.e.\ on~} \Gamma,
\end{align*}
for all $t \ge T_1$. From the separation property \eqref{esepe}, (A1) and \eqref{est34}, we can obtain
\begin{align*}
	 \bigl \|\beta  \bigl( u(t)  \bigr)  \bigr \|_{V}
	+  \bigl \|\beta_\Gamma  \bigl( u_\Gamma(t)  \bigr) \bigr \|_{V_\Gamma}\leq C, \quad \text{for a.a.\ }t \ge T_1.
\end{align*}
Next, it follows from \eqref{est33}  that
\begin{equation*}
	 \bigl \| \mu (t) \bigr \|_{V} +  \bigl \| \mu _\Gamma (t)  \bigr \|_{V_\Gamma }
	\le C, \quad \text{for a.a.\ }t \ge T_1.
\end{equation*}
As a consequence, we have 
\begin{equation*}
	 \bigl \| f(t) \bigr \|_{V} + \bigl \| f_\Gamma (t)  \bigr \|_{V_\Gamma }
	\le C, \quad \text{for a.a.\ }t \ge T_1,
\end{equation*}
which together with the elliptic regularity theory (see Lemma \ref{esLepH2}) 
yields the uniform estimate \eqref{est42}. 
\end{proof}

\begin{remark}
By \eqref{est32}, \eqref{est42} and interpolation, we easily see that $\boldsymbol{u}\in C([t,t+1]; \boldsymbol{W})$ for all $t\geq T_1$, hence, 
\begin{equation*}
	\boldsymbol{u}\in C \bigl( [T_1,\infty); \boldsymbol{W} \bigr).
\end{equation*} 
The uniform estimate \eqref{est42} and the compact embedding $H^3(\Omega)\times H^3(\Gamma)\hookrightarrow\hookrightarrow W\times W_\Gamma$ also imply that the $\omega$-limit set $\omega (\boldsymbol{u}_0)$ is compact in
$\boldsymbol{W}$ (an alternative proof for this fact is due to \eqref{omeS} and Lemma \ref{solsta}). 
\end{remark}

\subsection{Convergence to equilibrium}

Since $\omega (\boldsymbol{u}_0)$ is nonempty and compact in $\boldsymbol{W}$, we immediately have the sequent  convergence 
\begin{equation*}
	\lim_{t\to \infty} \mathrm{dist}  \bigl ( 
	S(t) \boldsymbol{u}_0, \omega(\boldsymbol{u}_0) 
	 \bigr)=0\quad \text{in } \boldsymbol{W}.
\end{equation*}
Our aim is to prove that for any initial datum$\boldsymbol{u}_0$ satisfying (A4), the corresponding $\omega $-limit set consists only one point, namely, there exists $\boldsymbol{u}_\infty\in \mathcal{S}_{m_0}$ such that
\begin{equation*}
	\boldsymbol{u}(t) \to \boldsymbol{u}_\infty, \quad {\rm in}~\boldsymbol{W} \quad  \text{as }t \to \infty.
\end{equation*}
This can be achieved by using the well-known {\L}ojasiewicz--Simon approach, see for instance, \cite{HJ99, Jen98}, and further applications \cite{ASS18,AW07,CFP06,GW08,GMS09,GMS11,GGM17,PW06,RH99,SW10,Wu07,Wu07b,WZ04}. 

The main tool is following extended {\L}ojasiewicz--{S}imon inequality.
Let $\boldsymbol{\psi }=(\psi ,\psi _\Gamma )\in \mathcal{S}_{a}$, $ a \in (-1,1)$. 
It is straightforward to verify that $\boldsymbol{\psi }$ is a critical point of the free energy $E$ (see \eqref{ene}). Moreover, we obtain the following lemma:
\begin{lemma}\label{LSi}
Suppose that (A1)--(A3) are satisfied. In addition, we assume that
 $\beta$, $\beta_\Gamma$ are real analytic on $(-1,1)$ and 
 $\pi $, $\pi_\Gamma$ are real analytic on 
 $\mathbb{R}$. 
 Let $\boldsymbol{\psi }=(\psi ,\psi _\Gamma )\in \mathcal{S}_{a}$, $ a\in (-1,1)$. There
exist constants $\theta ^* \in (0, 1/2)$ and $b^*>0$ such that
\begin{equation}
	\left\|
	\boldsymbol{P}
	\left(
	\begin{array}{c}
	 -\Delta w+\beta (w) + \pi (w) \\
	 \partial _{\boldsymbol{\nu }} w
	 -\Delta _\Gamma w_\Gamma +\beta_\Gamma (w_\Gamma )+\pi_\Gamma (w_\Gamma )
	\end{array}
	\right)
	\right\| _{\boldsymbol{H}_0}
	\ge
	\bigl| E(\boldsymbol{w})-E(\boldsymbol{\psi })\bigr|^{1-\theta ^*}
	\label{ls}
\end{equation}
for all
$\boldsymbol{w} \in \boldsymbol{W}$ satisfying
$\|\boldsymbol{w}-\boldsymbol{\psi }\|_{\boldsymbol{W}}<b^*$ and $m(\boldsymbol{w})= a$.
\end{lemma}

Lemma \ref{solsta} implies that all elements of $\mathcal{S}_{a}$ are uniformly separated from $\pm 1$. Then we can take $b^*>0$ sufficiently small such that any 
element $\boldsymbol{w} \in \boldsymbol{W}$ satisfying
$\|\boldsymbol{w}-\boldsymbol{\psi }\|_{\boldsymbol{W}}<b^*$ is uniformly separated from $\pm 1$. In particular, this choice prevents the possible singularity in the nonlinearities $\beta$, $\beta_\Gamma$. Keeping this fact in mind, we can follow the standard argument like in \cite{Jen98,RH99} to prove Lemma \ref{LSi}. More related to our problem \eqref{GMS}, we refer to \cite{GW08} for the case with mass conservation and a linear boundary condition, and to \cite{Wu07,SW10} for the case with nonlinear boundary condition but without mass conservation. When singular potential is considered, we refer to \cite{GMS11}.\medskip

\textbf{Proof of Theorem \ref{convthm}}. 
We now have all the necessary ingredients for the proof: 
\begin{itemize}
\item the characterization  of $\omega(\boldsymbol{u}_0)$;
\item the energy identity \eqref{equality};
\item the {\L}ojasiewicz--{S}imon type inequality \eqref{ls}.
\end{itemize}

The proof of  Theorem \ref{convthm} can be carried out in the same way as for instance, \cite[Section~2.4]{GW08}. 
We just would like to mention that in Lemma \ref{LSi}, if $\boldsymbol{w}$ is taken to be the weak solution $\boldsymbol{u}(t)$ of problem \eqref{GMS}  that can be shown falling into the small $\boldsymbol{W}$-neighborhood of a cluster point $\boldsymbol{u}_\infty\in \omega(\boldsymbol{u}_0)$ (which is indeed true for sufficiently large time), then by the generalized Poincar\'{e} inequality (Lemma  \ref{Poinc}), we have 
\begin{equation*}
	\|\boldsymbol{u}'\|_{\mathcal{V}_{\sigma,0}^*}=\|\boldsymbol{P} \boldsymbol{\mu}\| _{\mathcal{V}_{\sigma,0}}
	\ge
	C\bigl| E(\boldsymbol{u})-E(\boldsymbol{u}_\infty) \bigr|^{1-\theta ^*}.
\end{equation*}
This connects the energy dissipation in \eqref{equality} and the {\L}ojasiewicz--{S}imon inequality \eqref{ls} that leads to the proof. The rest of details are omitted. 
\hfill $\Box$

%%%%%%%%%%%%%%%%%%%%%%%%%%%%%%%%%%%
\appendix
\section{Auxiliary Lemmas}
\setcounter{equation}{0}

We report some lemmas that have been used in this paper.

\begin{lemma} \label{inter}\mbox{{\rm \cite[Lemma~5.1]{Lio68}}}
Let $B_0, B, B_1$ be three {B}anach spaces so that $B_0$ and $B_1$ are reflexive. Moreover,
$B_0
\mathop{\hookrightarrow} \mathop{\hookrightarrow}
B
\subset
B_1$. Then, for each $\delta >0$, there exists a positive constant
$C_\delta $ depends on $\delta $ such that
\begin{equation*}
	\|z\|_B \le \delta \|z\|_{B_0}+C_\delta \|z\|_{B_1}, \quad {\it for~all~}z \in B_0.
\end{equation*}
\end{lemma}

\begin{lemma} \label{gronwall} \mbox{{\rm\cite[Lemma~1.1]{Tem97} }}
Let $g,h,y$ be three positive locally integrable functions on $(t_0, \infty )$ such that $y'$ is locally integrable on $(t_0, \infty )$ and which satisfy
\begin{gather*}
	\frac{dy}{dt} \le gy +h,
	\\
	\int_{t}^{t+r} g(s)ds \le a_1, \quad
	\int_{t}^{t+r} h(s) ds \le a_2, \quad
	\int_{t}^{t+r} y(s) ds \le a_3 \quad {\it for~all~} t \ge t_0,
\end{gather*}
where $r, a_1, a_2, a_3$ are positive constants. Then
\begin{equation*}
	y(t+r) \le \left( \frac{a_3}{r} + a_2 \right) e^{a_1} \quad {\it for~all~} t \ge t_0.
\end{equation*}
\end{lemma}

\begin{lemma} \label{Poinc} \mbox{{\rm\cite[Section 2]{Gal08} and \cite[Lemma A]{CF15} }}
For every $\sigma\geq 0$, the following generalized Poincar\'{e} inequality holds 
\begin{equation}
	\| \boldsymbol{z} \|_{\boldsymbol{H}}^2
	\leq C_{\rm P} \|\boldsymbol{z}\|_{\mathcal{V}_{\sigma,0}}^2,\quad\text{for all } \boldsymbol{z}\in \mathcal{V}_{\sigma,0},\label{poin}
\end{equation}
for some constant $C_{\rm P}$ independent of $\boldsymbol{z}$.
\end{lemma}

\begin{lemma} \label{young} \mbox{{\rm\cite[Section~8.3]{AF03} }}
Let $f(t):=e^t-t-1$, $\widetilde{f}(s):=(1+s)\ln (1+s)-s$, then
\begin{equation*}
	st \le f(t)+\widetilde{f}(s)
	\quad {\it for~all~} s,t \ge 0.
\end{equation*}
\end{lemma}

\begin{lemma} \label{TM} \mbox{{\rm \cite[Theorem~2.2]{NSY97}}}
Let $\Omega \subset \mathbb{R}^2$ be a bounded domain with smooth boundary.
Then, there exists a positive constant $C_{\rm TM}$ such that
\begin{equation*}
	\int_{\Omega }^{} e^{|z|} dx \le C_{\rm TM}
	e^{C_{\rm TM} \|z\|_{V}^2}
	\quad {\it for~all~} z \in V.
\end{equation*}
\end{lemma}

\begin{lemma}\label{esLepH2} \mbox{{\rm\cite[Corollary A.1]{MZ05} }}
Let $\kappa>0$, $\Omega\subset \mathbb{R}^n$ ($n=2,3$) be a bounded domain with smooth boundary $\Gamma$. Consider the following linear elliptic problem
\begin{equation}
	\left\{
	\begin{aligned}
	&-\Delta \phi=h_1,& \text{ a.e.\ in}\ \Omega,\\
	&\phi|_\Gamma=\psi,& \text{ a.e.\ on}\ \Gamma,\\
	&-\kappa\Delta_\Gamma \psi+\psi+\partial_{\boldsymbol{\nu}}\phi=h_2, & \text{ a.e.\ on}\ \Gamma,\\
	\end{aligned}
	\right.
\label{Lep}
\end{equation}
where $(h_1, h_2)\in H^s(\Omega)\times H^s(\Gamma)$ for any $s\geq 0$ and $s+ 1/2\notin \mathbb{N}$.
Then every solution $(\phi, \psi)$ to problem \eqref{Lep} satisfies the following estimate
\begin{align}
\|\phi\|_{H^{s+2}(\Omega)}+\|\psi\|_{H^{s+2}(\Gamma)}\leq C 
 \bigl( 
\|h_1\|_{H^s(\Omega)}+\|h_2\|_{H^s(\Gamma)}
 \bigr),
\label{esLeph2}
\end{align}
for some constant $C>0$ that may depend on $\kappa$, $s$, $\Omega$ and $\Gamma$, but is independent of  $(\phi, \psi)$.
\end{lemma}
\medskip

%%%%%%%%%%%%%%%%%%%%%%%%%%%%%%%%%%%%%%%%%%%%%%%%
\section*{Acknowledgement}
T. Fukao acknowledges
the support from the JSPS KAKENHI Grant-in-Aid for Scientific Research(C), Japan
Grant Number 17K05321. H. Wu was partially supported by NNSFC Grant No. 11631011 and the Shanghai Center for Mathematical Sciences at Fudan University.

\end{document}